\documentclass{amsart}

\usepackage[english]{babel}

\usepackage[letterpaper,top=2cm,bottom=2cm,left=3cm,right=3cm,marginparwidth=1.75cm]{geometry}

\usepackage{amsmath, amssymb, amsthm}
\usepackage{graphicx,import}
\usepackage[colorlinks=true, allcolors=blue]{hyperref}
\usepackage{cleveref}

\usepackage{ifthen}
\usepackage{tikz,pgf,pgfmath}
\usetikzlibrary{shapes,calc,decorations.pathmorphing}
\pgfdeclarelayer{background}
\pgfdeclarelayer{middle}
\pgfsetlayers{background,middle,main}
\tikzstyle{edge}=[line width=.75pt]
\tikzstyle{fnode}=[fill=black,draw=black,circle,scale=\s]
\tikzstyle{pathnode}=[inner sep=.9pt]

\newcommand{\red}{red!50!white}
\newcommand{\blue}{blue!50!white}

\newcommand{\green}{green!50!gray}
\newcommand{\gray}{white!50!gray}
\tikzstyle{facet}=[fill=\green,fill opacity=0.6]
\tikzstyle{facetR}=[fill=\red,fill opacity=0.8]
\tikzstyle{facetL}=[fill=\blue,fill opacity=0.8]
\tikzstyle{pathnode}=[inner sep=.9pt]
\newcounter{i}
\newcounter{n}
\newcounter{m}
\newcounter{curx}
\newcounter{cury}
\newcounter{orient}
\newcommand{\nuTree}[7]{
	\begin{tikzpicture}
		\def\dx{#5};
		\def\s{3*\dx};	
		\setcounter{n}{0}
		\setcounter{m}{0}
		\setcounter{orient}{0}
		\setcounter{curx}{0}
		\setcounter{cury}{0}
		\foreach \step in {#1}{
			\ifthenelse{\equal{\theorient}{0}}{
				\addtocounter{n}{\step}
				\setcounter{orient}{1}
			}{
				\addtocounter{m}{\step}
				\setcounter{orient}{0}
			}
		}
		\setcounter{i}{0}
		\setcounter{orient}{0}
		\foreach \step in {#1}{
			\ifthenelse{\equal{\theorient}{0}}{
				\coordinate(q\thei) at (\thecurx*\dx,\thecury*\dx);
				\stepcounter{i}
				\coordinate(q\thei) at (\thecurx*\dx,{(\thecury+\step)*\dx});
				\stepcounter{i}
				\draw[\gray,dashed](\thecurx*\dx,\thecury*\dx) -- (\thecurx*\dx,{(\thecury+\step)*\dx});
				\begin{pgfonlayer}{background}
					\foreach \k in {1,...,\step}{
						\draw[\gray,dashed](0*\dx,{(\thecury+\k)*\dx}) -- (\thecurx*\dx,{(\thecury+\k)*\dx});
					}
				\end{pgfonlayer}
				\addtocounter{cury}{\step}
				\setcounter{orient}{1}
			}{
				\draw[\gray,dashed](\thecurx*\dx,\thecury*\dx) -- ({(\thecurx+\step)*\dx},\thecury*\dx);
				\begin{pgfonlayer}{background}
					\foreach \k in {0,...,\step}{
						\ifthenelse{\equal{\k}{\step}}{
						}{
							\draw[\gray,dashed]({(\thecurx+\k)*\dx},\thecury*\dx) -- ({(\thecurx+\k)*\dx},\then*\dx);
						}
					}
				\end{pgfonlayer}
				\addtocounter{curx}{\step}
				\setcounter{orient}{0}
			}
		}
		\ifthenelse{\equal{#7}{}}{}{
			\coordinate(zero) at (0*\dx,0*\dx);
			\coordinate(q\thei) at(0*\dx,\then*\dx);
			\addtocounter{i}{-1}
			\foreach \k in {0,2,...,\thei}{
				\pgfmathparse{\k+1}
				\let\res\pgfmathresult
				\begin{pgfonlayer}{middle}
					\fill[#7] (zero |- q\k) -- (q\k) -- (q\res) -- (zero |- q\res) -- cycle;
				\end{pgfonlayer}
			}
		}
		\setcounter{i}{1}
		\foreach \a/\b/\c/\d in {#2}{
			\draw(\a*\dx,\b*\dx) node[fnode,draw=\c,fill=\c](p\thei){};
			\ifthenelse{\equal{\d}{d}}{
				\draw[orange!50!gray,decorate,decoration={snake,amplitude=.5,segment length=2.5}](p\thei) -- (\a*\dx,{(\then+.5)*\dx});
			}{
			}
			\stepcounter{i}
		}
		\foreach \a/\b in {#3}{
			\draw[line width=3*\dx pt](\a) -- (\b);
		}
		\foreach \a/\b/\c in {#4}{
			\draw[\c,line width=.5pt](\a*\dx,\b*\dx) circle(.25*\dx);
		}
		\foreach \a/\b/\c/\d in {#6}{
			\draw[\d](\a*\dx,\b*\dx) node[scale=.75,anchor=west]{\c};
		}
	\end{tikzpicture}
}
\newcommand{\nuPath}[7]{
	\begin{tikzpicture}
		\def\dx{#4};
		\setcounter{n}{0}
		\setcounter{m}{0}
		\setcounter{orient}{0}
		\setcounter{curx}{0}
		\setcounter{cury}{0}
		\foreach \step in {#1}{
			\ifthenelse{\equal{\theorient}{0}}{
				\addtocounter{n}{\step}
				\setcounter{orient}{1}
			}{
				\addtocounter{m}{\step}
				\setcounter{orient}{0}
			}
		}
		\setcounter{i}{0}
		\setcounter{orient}{0}
		\foreach \step in {#1}{
			\ifthenelse{\equal{\theorient}{0}}{
				\coordinate(q\thei) at (\thecurx*\dx,\thecury*\dx);
				\stepcounter{i}
				\coordinate(q\thei) at (\thecurx*\dx,{(\thecury+\step)*\dx});
				\stepcounter{i}
				\draw[\gray,dashed](\thecurx*\dx,\thecury*\dx) -- (\thecurx*\dx,{(\thecury+\step)*\dx});
				\begin{pgfonlayer}{background}
					\foreach \k in {1,...,\step}{
						\draw[\gray,dashed](0*\dx,{(\thecury+\k)*\dx}) -- (\thecurx*\dx,{(\thecury+\k)*\dx});
					}
				\end{pgfonlayer}
				\addtocounter{cury}{\step}
				\setcounter{orient}{1}
			}{
				\draw[\gray,dashed](\thecurx*\dx,\thecury*\dx) -- ({(\thecurx+\step)*\dx},\thecury*\dx);
				\begin{pgfonlayer}{background}
					\foreach \k in {0,...,\step}{
						\ifthenelse{\equal{\k}{\step}}{
						}{
							\draw[\gray,dashed]({(\thecurx+\k)*\dx},\thecury*\dx) -- ({(\thecurx+\k)*\dx},\then*\dx);
						}
					}
				\end{pgfonlayer}
				\addtocounter{curx}{\step}
				\setcounter{orient}{0}
			}
		}
		\ifthenelse{\equal{#6}{}}{}{
			\coordinate(zero) at (0*\dx,0*\dx);
			\coordinate(q\thei) at(0*\dx,\then*\dx);
			\addtocounter{i}{-1}
			\foreach \k in {0,2,...,\thei}{
				\pgfmathparse{\k+1}
				\let\res\pgfmathresult
				\begin{pgfonlayer}{background}
					\fill[#6] (zero |- q\k) -- (q\k) -- (q\res) -- (zero |- q\res) -- cycle;
				\end{pgfonlayer}
			}
		}
		\setcounter{i}{0}
		\setcounter{cury}{0}
		\foreach \a in {#2}{
			\ifthenelse{\equal{\thei}{0}}{
				\draw[line width=2pt,rounded corners](\thei*\dx,\thei*\dx) -- (\thei*\dx,{\thei+\a)*\dx});
			}{
				\draw[line width=2pt,rounded corners]({(\thei-1)*\dx},\thecury*\dx) -- (\thei*\dx,\thecury*\dx);
				\draw[line width=2pt,rounded corners](\thei*\dx,\thecury*\dx) -- (\thei*\dx,\a*\dx);
				\ifthenelse{\equal{#7}{1}}{
					\draw[orange!50!gray,decorate,decoration={snake,amplitude=.5,segment length=2.5}]({(\thei-1)*\dx},\thecury*\dx) -- ({(\thei-1)*\dx},{(\then+.5)*\dx});
				}{}
			}
			\setcounter{cury}{\a}
			\stepcounter{i}
		}
		\foreach \a/\b/\c/\d in {#3}{
			\ifthenelse{\equal{\c}{}}{}{
				\draw[\c,line width=.5pt](\a*\dx,\b*\dx) circle(.25*\dx);
			}
			\fill[\d,line width=.5pt](\a*\dx,\b*\dx) circle(.15*\dx);
		}
		\foreach \a/\b/\c/\d in {#5}{
			\draw[\d](\a*\dx,\b*\dx) node[scale=.9,anchor=west]{\c};
		}
	\end{tikzpicture}
}

\newtheorem{theorem}{Theorem}[section]
\newtheorem{lemma}[theorem]{Lemma}
\newtheorem{corollary}[theorem]{Corollary}

\newtheorem{proposition}[theorem]{Proposition}

\theoremstyle{definition}
\newtheorem{definition}[theorem]{Definition}
\newtheorem{remark}[theorem]{Remark}

\newcommand{\Dyck}[1]{\operatorname{Dyck}_{#1}}
\DeclareMathOperator{\alt}{\operatorname{alt}}
\DeclareMathOperator{\elev}{\operatorname{elev}}
\DeclareMathOperator{\flush}{\operatorname{flush}}
\newcommand{\Tam}[1]{\operatorname{Tam}_{#1}}
\newcommand{\TamTrees}[1]{\operatorname{Tam}^{tr}_{#1}}
\newcommand{\altTam}[2]{\operatorname{Tam}_{#1}(#2)}
\newcommand{\altTamTrees}[2]{\operatorname{Tam}^{tr}_{#1}(#2)}
\newcommand{\drot}[1]{\lessdot_{#1}}
\newcommand{\reverse}[1]{\overleftarrow{#1}}
\newcommand{\nurev}{\overleftarrow{\nu}}
\newcommand{\hflushing}[2]{f^h_{#1,#2}}
\newcommand{\vflushing}[2]{f^v_{#1,#2}}

	\usepackage{todonotes}

	\title[Alt $\nu$-Tamari lattices]{On linear intervals in the alt $\nu$-Tamari lattices}
	\date{May 3, 2023}
	
	\author[C.~Ceballos]{Cesar Ceballos$^{\diamond}$} 
	\address[C.~Ceballos]{Institute of Geometry, TU Graz, Graz, Austria}
	\email{cesar.ceballos@tugraz.at}
	\urladdr{http://www.geometrie.tugraz.at/ceballos/}
	
	\author[C.~Chenevière]{Clément Chenevière$^{\star}$} 
	\address[C.~Chenevière]{Institut de Recherche Mathématique Avancée, Université de Strasbourg, Strasbourg, France\\
		and Ruhr-Universität Bochum, Bochum, Germany}
	\email{ccheneviere@unistra.fr}
	\urladdr{https://irma.math.unistra.fr/~cheneviere/}
	
	\thanks{This project was supported by the ANR-FWF International Cooperation Project PAGCAP, funded by the ANR Project
		ANR-21-CE48-0020 and the FWF Project I 5788. Cesar Ceballos was also supported by the Austrian Science Fund FWF,
		Project P 33278.}
	
 
\begin{document}
		\maketitle
		
		\begin{abstract}
			Given a lattice path $\nu$, the $\nu$-Tamari lattice and the $\nu$-Dyck lattice are two natural examples of partial order structures on the set of lattice paths that lie weakly above $\nu$.
			In this paper, we introduce a more general family of lattices, called alt $\nu$-Tamari lattices, which contains these two examples as particular cases. Unexpectedly, we show that all these lattices have the same number of linear intervals.
		\end{abstract}
		
		\section{Introduction}
		
		The classical Tamari lattice is a partial order on Catalan objects which has inspired a vast amount of research in various mathematical fields~\cite{tamari_festschrift}. One direction of research which has received a lot of attention in recent years regards its number of intervals~\cite{chapoton_tamari_intervals_2005}, which is conjectured to be equal to the dimension of the alternating component of an $\mathfrak S_n$-module in the study of trivariate diagonal harmonics~\cite{haiman_conjectures_1994}. 
		Motivated by this intriguing connection, Bergeron introduced a generalization of the Tamari lattice called the $m$-Tamari lattice, and conjectured that its number of intervals again coincides with the dimension of the alternating component of an $\mathfrak S_n$-module in higher trivariate diagonal 
		harmonics~\cite{bergeron_higher_2012}. 
		A formula for their enumeration and connections to representation theory can be found in~\cite{bousquet_representation_2013,bousquet_number_2011}.
		A further generalization of the Tamari lattice, which includes the $m$-Tamari lattice, is the $\nu$-Tamari lattice introduced by Pr\'eville-Ratelle and Viennot~\cite{preville_nu_tamari_2017}. 
		These lattices are indexed by a lattice path $\nu$, and their number of intervals is connected to the enumeration of non-separable planar maps as shown in~\cite{fang_enumeration_2017}.

		Inspired by the enumeration of intervals in the classical Tamari lattice and its generalizations, and guided by computer experimentation, Chapoton proposed to study the enumeration of the simpler class of linear intervals (intervals which are chains). 
		This led to the work of the second author in~\cite{cheneviere_linear_2022}, where he provides an explicit simple formula for the number of linear intervals in the classical Tamari lattice, and shows that their enumeration coincides with the enumeration of linear intervals in the Dyck lattice. 
		The Dyck lattice, sometimes called the Stanley lattice, is perhaps the most natural poset on Dyck paths, defined by $P\leq Q$ if $Q$ is weakly above $P$.
		In~\cite{cheneviere_linear_2022}, the author also defines a new family of posets called alt Tamari posets, which contain the Tamari lattice and the Dyck lattice as particular cases. He shows that all alt Tamari posets have the same number of linear intervals of any given length.  
		
		In this paper, we generalize the results in~\cite{cheneviere_linear_2022} by introducing a new family of posets called alt $\nu$-Tamari posets. We show that they are lattices, and that they all have the same number of linear intervals of any given length. 
		Figure~\ref{fig_nuTamari_nDyck_ENEENN} and Figure~\ref{fig_alt_nu_tamari_ENEENN_delta_one} illustrate the three different alt $\nu$-Tamari lattices for $\nu=ENEENN$. In each case, the number of linear intervals of length~$k$ is given by $\ell_k$ where $\ell=(\ell_0,\ell_1,\ell_2,\ell_3)=(16,24,16,3)$. For instance,~16 represents the trivial intervals of length 0, which are just the elements of each poset; there are 24 linear intervals of length 1, which correspond to the cover relations (edges in the figures); there are 16 linear intervals of length 2, and 3 linear intervals of length~3. The fact that these numbers coincide is somewhat surprising, since the posets look quite different. As a warm up exercise, the reader is invited to find the 3 linear intervals of length 3 in each of the figures.   
		Figure~\ref{fig_altnu_lattices_ENEEN_paths} illustrates the three different alt $\nu$-Tamari lattices for $\nu=ENEEN$.
		
		\begin{figure}[h]
			\begin{center}
				\begin{tikzpicture}%
	[x={(0cm, 4cm)},
	y={(-3.7cm, 0cm)},
	z={(-3cm, 6cm)},
	scale=0.27,
	back/.style={loosely dotted, thin},
	edge/.style={color=blue!95!black, thick},
	facet/.style={fill=red!95!black,fill opacity=0.800000},
	vertex/.style={inner sep=1pt,circle,draw=green!25!black,fill=green!75!black,thick,anchor=base}]
\scriptsize


\coordinate (c2311) at (0,0,0);
\coordinate (c1411) at (1,0,0);
\coordinate (c1321) at (2,1,0);
\coordinate (c1312) at (3,2,0);
\coordinate (c1222) at (3,4,0);
\coordinate (c1213) at (3,6,0);
\coordinate (c2221) at (0,3,0);
\coordinate (c1231) at (2,3,0);
\coordinate (c2212) at (0,6,0);
\coordinate (c1132) at (3,4,1);
\coordinate (c1123) at (3,6,1);
\coordinate (c2131) at (0,3,1);
\coordinate (c1141) at (2,3,1);
\coordinate (c2122) at (0,6,1);
\coordinate (c2113) at (0,7,1);
\coordinate (c1114) at (3,7,1);

\node (n2311) at (c2311) {1200};
\node (n1411) at (c1411) {0300};
\node (n1321) at (c1321) {0210};
\node (n1312) at (c1312) {0201};
\node (n1222) at (c1222) {0111};
\node (n1213) at (c1213) {0102}; 
\node (n2221) at (c2221) {1110};
\node (n1231) at (c1231) {0120};
\node (n2212) at (c2212) {1101};
\node (n1132) at (c1132) {0021};
\node (n1123) at (c1123) {0012}; 
\node (n2131) at (c2131) {1020};
\node (n1141) at (c1141) {0030};
\node (n2122) at (c2122) {1011};
\node (n2113) at (c2113) {1002};
\node (n1114) at (c1114) {0003};

\def\listEdgesFront{
n2311/n1411, n1411/n1321, n1321/n1312, n1312/n1222, n1222/n1213, n1321/n1231,
n2311/n2221, n2221/n1231, n1231/n1222, n2221/n2212, n2212/n1213,
n1141/n1132, n1132/n1123,  n2122/n1123,
n1213/n1123, n1222/n1132, n2212/n2122,
n2122/n2113, n2113/n1114, n1123/n1114}
;
\foreach \x/\y in \listEdgesFront{
    \draw[edge,->] (\x) -- (\y);
}

\def\listEdgesBack{
n2131/n1141, n2131/n2122,
n2221/n2131, n1231/n1141};
\foreach \x/\y in \listEdgesBack{
    \draw[edge,back,->] (\x) -- (\y);
}

\end{tikzpicture}
				\begin{tikzpicture}%
	[x={(0cm, 4cm)},
	y={(-3.5cm, 0cm)},
	z={(-2cm, 3cm)},
	scale=0.36,
	back/.style={loosely dotted, thin},
	edge/.style={color=blue!95!black, thick},
	facet/.style={fill=red!95!black,fill opacity=0.800000},
	vertex/.style={inner sep=1pt,circle,draw=green!25!black,fill=green!75!black,thick,anchor=base}]
\scriptsize


\coordinate (c2311) at (0,0,0);
\coordinate (c2221) at (0,0,1);
\coordinate (c2131) at (0,0,2);
\coordinate (c2212) at (0,1,1);
\coordinate (c2122) at (0,1,2);
\coordinate (c2113) at (0,2,2);
\coordinate (c1411) at (1,0,0);
\coordinate (c1321) at (1,0,1);
\coordinate (c1231) at (1,0,2);
\coordinate (c1141) at (1,0,3);
\coordinate (c1312) at (1,1,1);
\coordinate (c1222) at (1,1,2);
\coordinate (c1132) at (1,1,3);
\coordinate (c1213) at (1,2,2);
\coordinate (c1123) at (1,2,3);
\coordinate (c1114) at (1,3,3);

\node (n2311) at (c2311) {1200};
\node (n1411) at (c1411) {0300};
\node (n1321) at (c1321) {0210};
\node (n1312) at (c1312) {0201};
\node (n1222) at (c1222) {0111};
\node (n1213) at (c1213) {0102}; 
\node (n2221) at (c2221) {1110};
\node (n1231) at (c1231) {0120};
\node (n2212) at (c2212) {1101};
\node (n1132) at (c1132) {0021};
\node (n1123) at (c1123) {0012}; 
\node (n2131) at (c2131) {1020};
\node (n1141) at (c1141) {0030};
\node (n2122) at (c2122) {1011};
\node (n2113) at (c2113) {1002};
\node (n1114) at (c1114) {0003};

\def\listEdgesFront{
n2311/n2221, n2221/n2212, n2212/n2122, n2122/n2113,
n1411/n1321, n1321/n1312, n1312/n1222, n1222/n1213, n1213/n1123, n1123/n1114,
n1321/n1231, n1231/n1222, n1222/n1132, n1132/n1123,
n1231/n1141, n1141/n1132,
n2311/n1411, n2221/n1321, n2212/n1312, n2122/n1222, n2113/n1213}
;
\foreach \x/\y in \listEdgesFront{
    \draw[edge,->] (\x) -- (\y);
}

\def\listEdgesBack{
n2221/n2131, n2131/n2122, n2131/n1231};
\foreach \x/\y in \listEdgesBack{
    \draw[edge,back,->] (\x) -- (\y);
}

\end{tikzpicture}
			\end{center}
			\caption{The $\nu$-Tamari lattice and $\nu$-Dyck lattice for $\nu=ENEENN$. They are the alt~$\nu$-Tamari lattices $\altTam{\nu}{\delta}$ for $\delta=(2,0,0)$ and $\delta=(0,0,0)$, respectively.}
			\label{fig_nuTamari_nDyck_ENEENN}
		\end{figure}
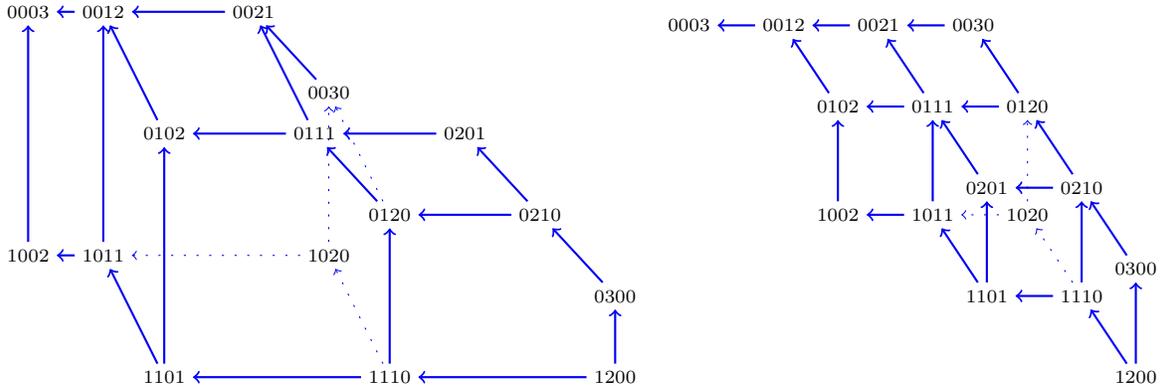
		
		\begin{figure}[h]
			\begin{center}
				\begin{tikzpicture}%
	[x={(0cm, 4cm)},
	y={(-5cm, 0cm)},
	z={(-2cm, 4cm)},
	scale=0.35,
	back/.style={loosely dotted, thin},
	edge/.style={color=blue!95!black, thick},
	facet/.style={fill=red!95!black,fill opacity=0.800000},
	vertex/.style={inner sep=1pt,circle,draw=green!25!black,fill=green!75!black,thick,anchor=base}]
\scriptsize


\coordinate (c2311) at (0,-1,0);
\coordinate (c1411) at (1,-1,0);
\coordinate (c2221) at (0,0,0);
\coordinate (c1321) at (1,0,0);
\coordinate (c1231) at (2,1,0);
\coordinate (c1141) at (2,2,0);
\coordinate (c2131) at (0,2,0);
\coordinate (c2212) at (0,0,1);
\coordinate (c1312) at (1,0,1);
\coordinate (c1222) at (2,1,1);
\coordinate (c1132) at (2,2,1);
\coordinate (c2122) at (0,2,1);
\coordinate (c1213) at (3,2,1);
\coordinate (c1123) at (3,3,1);
\coordinate (c1114) at (3,4,1);
\coordinate (c2113) at (0,4,1);

\node (n2311) at (c2311) {1200};
\node (n1411) at (c1411) {0300};
\node (n1321) at (c1321) {0210};
\node (n1312) at (c1312) {0201};
\node (n1222) at (c1222) {0111};
\node (n1213) at (c1213) {0102}; 
\node (n2221) at (c2221) {1110};
\node (n1231) at (c1231) {0120};
\node (n2212) at (c2212) {1101};
\node (n1132) at (c1132) {0021};
\node (n1123) at (c1123) {0012}; 
\node (n2131) at (c2131) {1020};
\node (n1141) at (c1141) {0030};
\node (n2122) at (c2122) {1011};
\node (n2113) at (c2113) {1002};
\node (n1114) at (c1114) {0003};

\def\listEdgesFront{
n2311/n1411, n1411/n1321, n2311/n2221,
n2221/n1321, n1321/n1231, n1231/n1141, n2221/n2131, n2131/n1141,
n1321/n1312, n1231/n1222, n1141/n1132, n2131/n2122,
n1312/n1222, n1222/n1132, n2122/n1132,
n1222/n1213, n1213/n1123, n1123/n1114, n2122/n2113, n2113/n1114, n1132/n1123}
;
\foreach \x/\y in \listEdgesFront{
    \draw[edge,->] (\x) -- (\y);
}

\def\listEdgesBack{
n2221/n2212,
n2212/n2122, n2212/n1312};
\foreach \x/\y in \listEdgesBack{
    \draw[edge,back,->] (\x) -- (\y);
}

\end{tikzpicture}
			\end{center}
			\caption{The alt $\nu$-Tamari lattice $\altTam{\nu}{\delta}$ for $\nu=ENEENN$ and $\delta=(1,0,0)$.}
			\label{fig_alt_nu_tamari_ENEENN_delta_one}
		\end{figure}
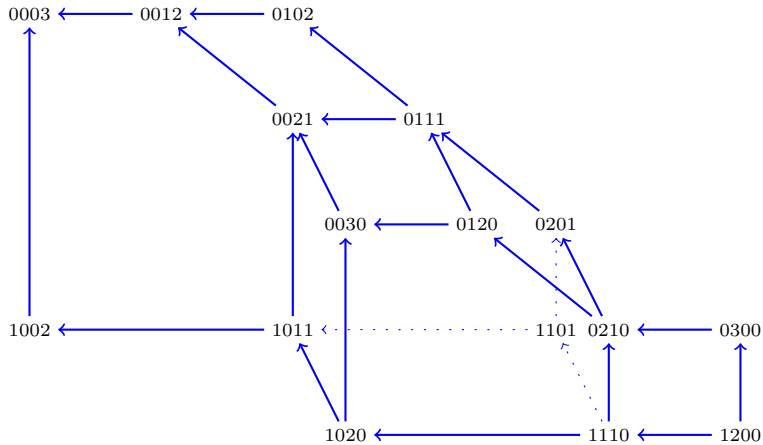
		
		As the figures suggest, the alt $\nu$-Tamari lattices possess a rich underlying geometric structure, which seems to be realizable as a polytopal complex in some Euclidean space. This was shown to be true for $\nu$-Tamari lattices in~\cite{ceballos_geometry_2019}, where polytopal complex realizations induced by some arrangements of tropical hyperplanes are provided. Similar geometric realizations in this general context for alt~$\nu$-Tamari lattices will be presented in forthcoming work~\cite{ceballos_altnuAssociahedra_2023}.
		
		\section{The $\nu$-Dyck lattice}
		
		Let $\nu$ be a lattice path on the plane from $(0,0)$ to $(m,n)$, consisting of a finite number of north and east unit steps.
		We may represent a path $\nu$ as a word in the letters $ E $ and $ N $ for east and north steps respectively.
		We may as well represent $\nu$ as a sequence of non negative integers $ (\nu_0, \nu_1, \dots, \nu_n) $, where $ n \in \mathbb{N} $ is the number of north steps of $\nu$, $ \nu_0 $ is the number of initial east steps, and $ \nu_i \geq 0$ is the number of consecutive east steps immediately following the $ i $-th north step of $\nu$.
		In particular, $m=\nu_0+\dots+\nu_n$ is the total number of east steps.  
		For instance, the path $ ENEENNENEEE $ would correspond to the sequence $ (1,2,0,1,3) $, while $ENEENN$ corresponds to (1,2,0,0).
		
		A \emph{$\nu$-path} $\mu$ is a lattice path using north and east steps, with the same endpoints as $\nu$, that is weakly above $\nu$.
		Alternatively, $\mu = (\mu_0, \dots, \mu_n)$ is a $\nu$-path if and only if $\sum_{i = 0}^{j} \mu_i \leq \sum_{i = 0}^{j} \nu_i $ for all $ 0 \leq j \leq n $, with equality for $j=n$.
		The elements of the posets in 
		Figure~\ref{fig_nuTamari_nDyck_ENEENN} and Figure~\ref{fig_alt_nu_tamari_ENEENN_delta_one} are labelled by $\nu$-paths using this representation, were we omit the commas and parentheses for simplicity. For instance, the label 1200 is the minimal  path $(1,2,0,0)$, which corresponds to $\nu=ENEENN$.   
		
		\begin{definition}
			The \emph{$\nu$-Dyck lattice} $ \Dyck{\nu} $ is the poset on $\nu$-paths where $ P \leq Q $ if $ Q $ is weakly above~$ P $.
		\end{definition}
		
		An example of the $\nu$-Dyck lattice for $\nu=ENEEN$ is illustrated on the left of~\Cref{fig_altnu_lattices_ENEEN_paths}.

		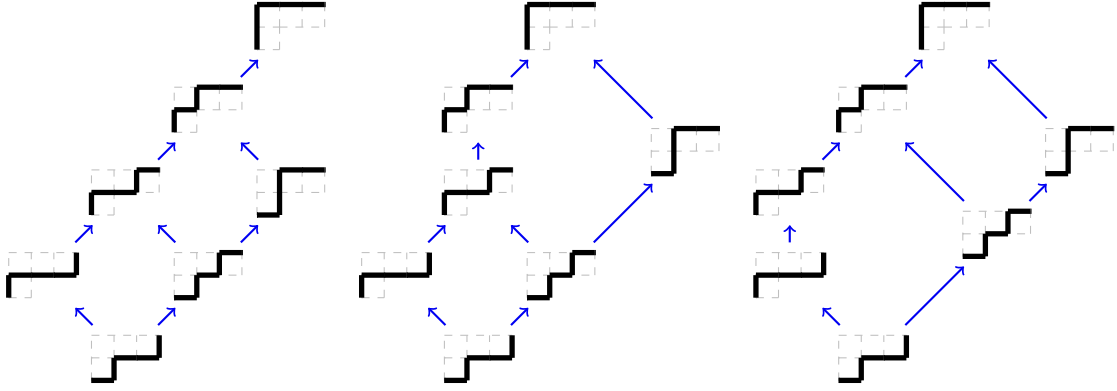
\begin{figure}[htb]
			\begin{center}
				\begin{tikzpicture}%
	[scale=1.1,
	x={(1cm, 0cm)},
	y={(0cm, 1cm)},
	back/.style={loosely dotted, thin},
	edge/.style={color=blue!95!black, thick},
	facet/.style={fill=red!95!black,fill opacity=0.800000},
	vertex/.style={inner sep=1pt,circle,draw=green!25!black,fill=green!75!black,thick,anchor=base}]
	\scriptsize
	
	\coordinate (c120) at (0,0);
	\coordinate (c030) at (-1,1);
	\coordinate (c021) at (0,2);
	\coordinate (c012) at (1,3);
	\coordinate (c003) at (2,4);
	
	\coordinate (c111) at (1,1);
	\coordinate (c102) at (2,2);
	\node (n120) at (c120) {\nuPath{0,1,1,2,1}{0,1,1,2}{}{.3}{}{}{}};
	\node (n030) at (c030) {\nuPath{0,1,1,2,1}{1,1,1,2}{}{.3}{}{}{}};
	\node (n021) at (c021) {\nuPath{0,1,1,2,1}{1,1,2,2}{}{.3}{}{}{}};
	\node (n012) at (c012) {\nuPath{0,1,1,2,1}{1,2,2,2}{}{.3}{}{}{}};
	\node (n003) at (c003) {\nuPath{0,1,1,2,1}{2,2,2,2}{}{.3}{}{}{}};
	
	\node (n111) at (c111) {\nuPath{0,1,1,2,1}{0,1,2,2}{}{.3}{}{}{}};
	\node (n102) at (c102) {\nuPath{0,1,1,2,1}{0,2,2,2}{}{.3}{}{}{}};
	
	\draw[edge,->] (n120) -- (n030); 
	\draw[edge,->] (n030) -- (n021); 
	\draw[edge,->] (n021) -- (n012); 
	\draw[edge,->] (n012) -- (n003); 
	
	\draw[edge,->] (n120) -- (n111);
	\draw[edge,->] (n111) -- (n021);
	\draw[edge,->] (n111) -- (n102);
	
	\draw[edge,->] (n102) -- (n012);
\end{tikzpicture}
				\begin{tikzpicture}%
	[scale=1.1,
	x={(1cm, 0cm)},
	y={(0cm, 1cm)},
	back/.style={loosely dotted, thin},
	edge/.style={color=blue!95!black, thick},
	facet/.style={fill=red!95!black,fill opacity=0.800000},
	vertex/.style={inner sep=1pt,circle,draw=green!25!black,fill=green!75!black,thick,anchor=base}]
	\scriptsize
	
	\coordinate (c120) at (0,0);
	\coordinate (c030) at (-1,1);
	\coordinate (c021) at (0,2);
	\coordinate (c012) at (0,3);
	\coordinate (c003) at (1,4);
	
	\coordinate (c111) at (1,1);
	\coordinate (c102) at (2.5,2.5);
	
	\node (n120) at (c120) {\nuPath{0,1,1,2,1}{0,1,1,2}{}{.3}{}{}{}};
	\node (n030) at (c030) {\nuPath{0,1,1,2,1}{1,1,1,2}{}{.3}{}{}{}};
	\node (n021) at (c021) {\nuPath{0,1,1,2,1}{1,1,2,2}{}{.3}{}{}{}};
	\node (n012) at (c012) {\nuPath{0,1,1,2,1}{1,2,2,2}{}{.3}{}{}{}};
	\node (n003) at (c003) {\nuPath{0,1,1,2,1}{2,2,2,2}{}{.3}{}{}{}};
	
	\node (n111) at (c111) {\nuPath{0,1,1,2,1}{0,1,2,2}{}{.3}{}{}{}};
	\node (n102) at (c102) {\nuPath{0,1,1,2,1}{0,2,2,2}{}{.3}{}{}{}};
	
	\draw[edge,->] (n120) -- (n030); 
	\draw[edge,->] (n030) -- (n021); 
	\draw[edge,->] (n021) -- (n012); 
	\draw[edge,->] (n012) -- (n003); 
	
	\draw[edge,->] (n120) -- (n111);
	\draw[edge,->] (n111) -- (n102);
	\draw[edge,->] (n102) -- (n003);
	
	\draw[edge,->] (n111) -- (n021);
\end{tikzpicture}
				\begin{tikzpicture}%
	[scale=1.1,
	x={(1cm, 0cm)},
	y={(0cm, 1cm)},
	back/.style={loosely dotted, thin},
	edge/.style={color=blue!95!black, thick},
	facet/.style={fill=red!95!black,fill opacity=0.800000},
	vertex/.style={inner sep=1pt,circle,draw=green!25!black,fill=green!75!black,thick,anchor=base}]
	\scriptsize
	
	\coordinate (c120) at (0,0);
	\coordinate (c030) at (-1,1);
	\coordinate (c021) at (-1,2);
	\coordinate (c012) at (0,3);
	\coordinate (c003) at (1,4);
	
	\coordinate (c111) at (1.5,1.5);
	\coordinate (c102) at (2.5,2.5);
	
	\node (n120) at (c120) {\nuPath{0,1,1,2,1}{0,1,1,2}{}{.3}{}{}{}};
	\node (n030) at (c030) {\nuPath{0,1,1,2,1}{1,1,1,2}{}{.3}{}{}{}};
	\node (n021) at (c021) {\nuPath{0,1,1,2,1}{1,1,2,2}{}{.3}{}{}{}};
	\node (n012) at (c012) {\nuPath{0,1,1,2,1}{1,2,2,2}{}{.3}{}{}{}};
	\node (n003) at (c003) {\nuPath{0,1,1,2,1}{2,2,2,2}{}{.3}{}{}{}};
	
	\node (n111) at (c111) {\nuPath{0,1,1,2,1}{0,1,2,2}{}{.3}{}{}{}};
	\node (n102) at (c102) {\nuPath{0,1,1,2,1}{0,2,2,2}{}{.3}{}{}{}};
	
	\draw[edge,->] (n120) -- (n030); 
	\draw[edge,->] (n030) -- (n021); 
	\draw[edge,->] (n021) -- (n012); 
	\draw[edge,->] (n012) -- (n003); 
	
	\draw[edge,->] (n120) -- (n111);
	\draw[edge,->] (n111) -- (n102);
	\draw[edge,->] (n102) -- (n003);
	
	\draw[edge,->] (n111) -- (n012);
\end{tikzpicture}
			\end{center}
			\caption{Examples of alt $\nu$-Tamari lattices $\altTam{\nu}{\delta}$ for $\nu=ENEEN=(1,2,0)$. Left: the $\nu$-Dyck lattice, for $\delta=(0,0)$. Middle: the lattice for $\delta=(1,0)$. Right: the $\nu$-Tamari lattice, for $\delta=(2,0)$.  \\
				In each case, the number of linear intervals of length~$k$ is given by $\ell_k$ where \mbox{$\ell=(\ell_0,\ell_1,\ell_2,\ell_3)=(7,8,4,1)$}. For instance,~7 represents the trivial intervals of length~0, which are just the elements of each poset; there are 8 linear intervals of length 1, which correspond to the edges; 4 linear interval of length 2, and 1 linear interval of length 3.
			}
			\label{fig_altnu_lattices_ENEEN_paths}
		\end{figure}
		
		\begin{remark}
			The case where $ \nu $ is $ (NE)^n $ coincides with the classical Dyck lattice on Dyck paths of size $ n $.
		\end{remark}
		
		\begin{remark}
			Covering relations $ P \lessdot Q $ in the $\nu$-Dyck lattice consist of transforming a valley $ EN $ in a peak $ NE $ in some $\nu$-path $ P $.
		\end{remark}
		
		\subsection{Left and right intervals in the $\nu$-Dyck lattice}
		
		We focus on the special class of linear intervals of a poset. An interval~$ [P,Q] $ is \emph{linear} if it is totally ordered, or equivalently if it is a chain of the form~\mbox{$ P = P_0 <P_1 < \dots < P_\ell = Q $}. The \emph{length} of such a linear interval is defined to be~$\ell$. A linear interval of length zero (containing only one element) is said to be \emph{trivial}. A linear interval of length one is by definition a covering relation since it contains two elements.
		The non trivial linear intervals of the $\nu$-Dyck lattice can be easily characterized into two different classes. 
		
		\begin{definition}
			An interval $ [P,Q] $ in $ \Dyck{\nu}  $ is a \emph{left interval} if $ Q $ is obtained from $ P $ by transforming a subpath $ E^\ell N $ into $ N E^\ell $ for some $ \ell \geq 1$.
			It is a \emph{right interval} if $ Q $ is obtained from $ P $ by transforming a subpath $ E N^\ell $ into $ N^\ell E $ for some $ \ell \geq 1$.
		\end{definition}

		\begin{lemma} \label{lem:Dyck-height-2}
			The linear intervals of length $ 2 $ are either left or right intervals.
		\end{lemma}
		
		\begin{proof}
			Let $ P  \lessdot Q  \lessdot R$ be a linear interval of length $ 2 $.
			The covering relations $ P \lessdot Q $ transforms a valley $ EN $ of $ P $ into a peak $ NE $.
			If the next covering relation $ Q \lessdot R $ happens at a valley of $ Q $ that is also a valley of $ P $, then the $ [P,R] $ is a square.
			Thus, this second covering relation must use either of the two steps of the peak $ NE $ that was created in $ Q $.
		\end{proof}
		
		\begin{proposition} \label{prop:linear_in_Dyck}
			The left and right intervals in the previous definition are linear intervals of length~$\ell$. Moreover,  all non trivial linear intervals in $ \Dyck{\nu}  $ are either left or right intervals.
		\end{proposition}
		
		\begin{proof}
			If $ [P,Q] $ is an interval of this form with $ \ell \geq 1 $, then there exists only one maximal chain from $ P $ to $ Q $.
			Indeed, there is only one valley of $ P $ that is not a valley of $ Q $ and thus, any maximal chain from $ P $ to $ Q $ starts at this valley.
			We then obtain an interval of the same form, but $ \ell $ has decreased by~$ 1 $ and we conclude by induction.
			
			\medskip
			
			The \Cref{lem:Dyck-height-2} proves that all linear intervals of height $ k = 2 $ are either left or right intervals.
			Suppose that $ [P,Q] $ is a linear interval of height $ k+1 \geq 3 $.
			It is linear so $ Q $ has only one lower cover $ Q' $ in $ [P,Q] $. 
			Then $ [P,Q'] $ is linear of height $ k $ and thus by induction, it is of the prescribed form.
			
			Suppose that $ Q' $ is obtained from $ P $ by transforming a subpath $ E^k N $ into $ N E^k $, which creates this peak $ NE $ in $ Q' $ followed by $ k-1 $ east steps.
			Then, \cref{lem:Dyck-height-2} ensures that the covering relation $ Q' \lessdot Q $ has to use the north step $ N $ of this peak and thus, $ Q $ is obtained from $ P $ by changing a subpath $ E^{k+1} N $ into $ N E^{k+1} $.
			
			Suppose now that  $ Q' $ is obtained from $ P $ by transforming a subpath $ E N^k $ into $ N^k E $.
			Similarly, \cref{lem:Dyck-height-2} ensures that $ Q' \lessdot Q $ has to use the east step $ E $ of this peak and $ Q $ is obtained from $ P $ by transforming a subpath $ E N^{k+1} $ into $ N^{k+1} E $.
		\end{proof}
		
		\begin{corollary}
			Left intervals of length $\ell$ in $\Dyck{\nu}$ are in bijection with $\nu$-paths marked at a north step preceded by $ \ell $ east steps.
			Right intervals of length $\ell$ in $\Dyck{\nu}$ are in bijection with $\nu$-paths marked at an east step followed by $ \ell $ north steps.
		\end{corollary}

		\section{The $\nu$-Tamari lattice}
		
		The $\nu$-Tamari lattices are a generalization of the Tamari lattice.
		They were defined in terms of $\nu$-paths by Préville-Ratelle and Viennot in~\cite{preville_nu_tamari_2017}. An alternative description in terms of $\nu$-trees was presented in~\cite{ceballos_nu_trees_2020}. 
		
		\subsection{On $\nu$-paths}	
		For a lattice point $ p $ on a $\nu$-path $\mu$, define its $\nu$-altitude $ \alt_\nu(p) $ to be the maximum number of horizontal steps that can be added to the right of $ p $ without crossing $\nu$.
		Given a valley $ EN $ of $\mu$, let $ p $ be the lattice point between the east and north steps.
		Let $ q $ be the next lattice point of $ \mu $ such that $ \alt_\nu(q) =  \alt_\nu(p) $, and $ \mu_{[p,q]} $ be the subpath of $\mu$ that starts at $ p $ and ends at $ q $.
		Let $\mu'$ be the path obtained from $\mu$ by switching $ \mu_{[p,q]} $ with the east step $E$ that precedes it.
		The $\nu$-rotation of~$\mu$ at the valley $ p $ is defined to be $ \mu \lessdot_\nu \mu' $. An example is illustrated in Figure~\ref{fig_nu_path_rotation}.
		
		\begin{figure}
			\begin{center}
				\centering
				\begin{minipage}[c]{.35\linewidth}
					\hfill
					\def\svgwidth{.7\linewidth}
					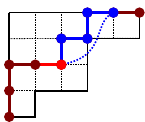
				\end{minipage}
				$\quad \lessdot_{\nu} \quad$
				\begin{minipage}[c]{.35\linewidth}
					\def\svgwidth{.7\linewidth}
					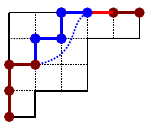
					\hfill
				\end{minipage}
			\end{center}
			\caption{The rotation operation of a $\nu$-path. Each node is labelled with its $\nu$-altitude.}
			\label{fig_nu_path_rotation}
		\end{figure}

		\begin{definition}
			The $\nu$-Tamari poset $\Tam{\nu}$ is the reflexive transitive closure of $\nu$-rotations on $\nu$-paths.
		\end{definition}
		
		An example of the $\nu$-Tamari poset for $\nu=ENEEN$ is illustrated on the right of~\Cref{fig_altnu_lattices_ENEEN_paths}. 
		
		\begin{theorem}[{\cite[Theorem~1]{preville_nu_tamari_2017}}]
			The $\nu$-Tamari poset is a lattice. The $\nu$-rotations are exactly its covering relations. 
		\end{theorem}
		
		Another approach to define the $\nu$-Tamari lattice is to introduce the notion of \emph{$\nu$-elevation} of a subpath as the difference of $\nu$-altitude between its ending point and its starting point.    
		We thus write $ \elev_\nu(E) = -1 $ for an east step $E$ and $\elev_\nu(N_i) = \nu_i$ if $ N_i $ is the $ i $-th north step of a $\nu$-path $\mu$.
		For any subpath $ A $ of $\mu$, we then have $\elev_\nu(A) = \sum_{a \in A} \elev_\nu(a)$ as the sum of the $\nu$-elevation of the steps of $A$.
		
		The \emph{$\nu$-excursion} of a north step $ N $ of a $\nu$-path $\mu$ is defined as the shortest subpath~$ A $ of $ \mu $ that starts with this $ N $ and such that $ \elev_\nu(A) = 0 $.
		It follows from the definition of the $\nu$-excursion that exchanging the east step $ E $ of a valley with the $\nu$-excursion that follows it is exactly a covering relation in $ \Tam{\nu} $.
		
		
		\subsection{On $\nu$-trees}\label{sec_nu_trees}
		
		One can also define a poset on $\nu$-trees which is isomorphic to the $\nu$-Tamari lattice.
		
		We denote by $F_{\nu}$ the Ferrers diagram that lies weakly above $\nu$ in the smallest rectangle containing~$\nu$.  Let~$L_{\nu}$ denote the set of lattice points inside $F_{\nu}$. 
		We say that two points~$p,q\in L_{\nu}$ are \empty{$\nu$-incompatible} if $p$ is strictly southwest or strictly northeast of $q$, and the smallest rectangle containing $p$ and $q$ lies entirely in~$F_{\nu}$.  Otherwise, $p$ and $q$ are said to be \empty{$\nu$-compatible}. 
		A \emph{$\nu$-tree} is a maximal collection of pairwise $\nu$-compatible elements in $L_\nu$. 
		In particular, the vertex at the top-left corner of $F_\nu$ is $\nu$-compatible with everyone else, and belongs to every $\nu$-tree. Connecting two consecutive  elements (not necessarily at distance $1$) in the same row or column allows us to visualize $\nu$-trees as classical rooted binary trees~\cite{ceballos_nu_trees_2020}. The vertex at top-left corner of $F_\nu$ is always the \emph{root}. An example of a $\nu$-tree and the rotation operation which we now describe is shown in \Cref{fig_nu_tree_rotation}.
		
		\begin{figure}
		
		\begin{center}
			\centering
			\begin{minipage}[c]{.4\linewidth}
				\hfill
				\def\svgwidth{.7\linewidth}
\begingroup%
  \makeatletter%
  \providecommand\color[2][]{%
    \errmessage{(Inkscape) Color is used for the text in Inkscape, but the package 'color.sty' is not loaded}%
    \renewcommand\color[2][]{}%
  }%
  \providecommand\transparent[1]{%
    \errmessage{(Inkscape) Transparency is used (non-zero) for the text in Inkscape, but the package 'transparent.sty' is not loaded}%
    \renewcommand\transparent[1]{}%
  }%
  \providecommand\rotatebox[2]{#2}%
  \newcommand*\fsize{\dimexpr\f@size pt\relax}%
  \newcommand*\lineheight[1]{\fontsize{\fsize}{#1\fsize}\selectfont}%
  \ifx\svgwidth\undefined%
    \setlength{\unitlength}{45.45938747bp}%
    \ifx\svgscale\undefined%
      \relax%
    \else%
      \setlength{\unitlength}{\unitlength * \real{\svgscale}}%
    \fi%
  \else%
    \setlength{\unitlength}{\svgwidth}%
  \fi%
  \global\let\svgwidth\undefined%
  \global\let\svgscale\undefined%
  \makeatother%
  \begin{picture}(1,0.78177034)%
    \lineheight{1}%
    \setlength\tabcolsep{0pt}%
    \put(0,0){\includegraphics[width=\unitlength,page=1]{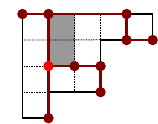}}%
    \put(0.27407806,0.72592261){\color[rgb]{0.50196078,0,0}\makebox(0,0)[rt]{\lineheight{1.25}\smash{\begin{tabular}[t]{r}$p$\end{tabular}}}}%
    \put(0.5050532,0.29696837){\color[rgb]{0.50196078,0,0}\makebox(0,0)[lt]{\lineheight{1.25}\smash{\begin{tabular}[t]{l}$r$\end{tabular}}}}%
    \put(0.27407805,0.29696837){\color[rgb]{1,0,0}\makebox(0,0)[rt]{\lineheight{1.25}\smash{\begin{tabular}[t]{r}$q=p\llcorner r$\end{tabular}}}}%
  \end{picture}%
\endgroup%

			\end{minipage}
			$\quad \lessdot_{\nu} \quad$
			\begin{minipage}[c]{.4\linewidth}
				\def\svgwidth{.7\linewidth}
\begingroup%
  \makeatletter%
  \providecommand\color[2][]{%
    \errmessage{(Inkscape) Color is used for the text in Inkscape, but the package 'color.sty' is not loaded}%
    \renewcommand\color[2][]{}%
  }%
  \providecommand\transparent[1]{%
    \errmessage{(Inkscape) Transparency is used (non-zero) for the text in Inkscape, but the package 'transparent.sty' is not loaded}%
    \renewcommand\transparent[1]{}%
  }%
  \providecommand\rotatebox[2]{#2}%
  \newcommand*\fsize{\dimexpr\f@size pt\relax}%
  \newcommand*\lineheight[1]{\fontsize{\fsize}{#1\fsize}\selectfont}%
  \ifx\svgwidth\undefined%
    \setlength{\unitlength}{45.38438632bp}%
    \ifx\svgscale\undefined%
      \relax%
    \else%
      \setlength{\unitlength}{\unitlength * \real{\svgscale}}%
    \fi%
  \else%
    \setlength{\unitlength}{\svgwidth}%
  \fi%
  \global\let\svgwidth\undefined%
  \global\let\svgscale\undefined%
  \makeatother%
  \begin{picture}(1,0.78306227)%
    \lineheight{1}%
    \setlength\tabcolsep{0pt}%
    \put(0,0){\includegraphics[width=\unitlength,page=1]{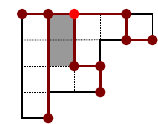}}%
    \put(0.50423531,0.72712225){\color[rgb]{1,0,0}\makebox(0,0)[lt]{\lineheight{1.25}\smash{\begin{tabular}[t]{l}$q'=p\urcorner r$\end{tabular}}}}%
    \put(0.27287847,0.72712225){\color[rgb]{0.50196078,0,0}\makebox(0,0)[rt]{\lineheight{1.25}\smash{\begin{tabular}[t]{r}$p$\end{tabular}}}}%
    \put(0.50423531,0.29745913){\color[rgb]{0.50196078,0,0}\makebox(0,0)[lt]{\lineheight{1.25}\smash{\begin{tabular}[t]{l}$r$\end{tabular}}}}%
  \end{picture}%
\endgroup%

				\hfill
			\end{minipage}
		\end{center}
			\caption{The rotation operation of a $\nu$-tree.}
			\label{fig_nu_tree_rotation}
		\end{figure}
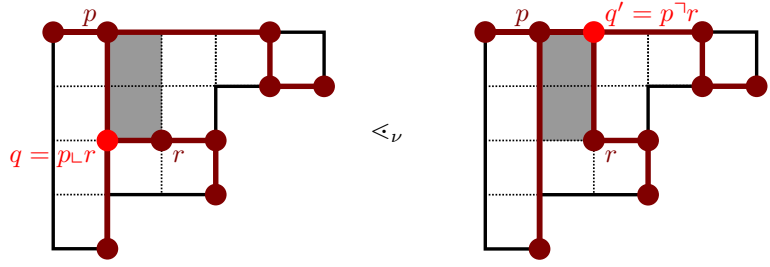     
		
		Let $T$ be a $\nu$-tree and $p,r\in T$ be two elements which do not lie in the same row or same column. We denote by $p\square r$ the smallest rectangle containing $p$ and $r$, and write~$p\llcorner r$~(resp. $p\urcorner r$) for the lower left corner (resp. upper right corner) of $p\square r$.  
		
		Let $p,q,r\in T$ be such that $q=p\llcorner r$ and no other elements besides $p,q,r$ lie in $p\square r$. 
		The \emph{$\nu$-rotation} of $T$ at $q$ is defined as the set $T'=\bigl(T\setminus\{q\})\cup\{q'\}$, where $q'=p\urcorner r$, and we write $T \drot{\nu} T'$.  
		As proven in~\cite[Lemma~2.10]{ceballos_nu_trees_2020}, the rotation of a $\nu$-tree is also a $\nu$-tree. 
		
		\begin{definition}
			The \emph{rotation poset of $\nu$-trees} $\TamTrees{\nu}$ is the reflexive transitive closure of $\nu$-rotations.   
		\end{definition}
		
		An example of the rotation poset of $\nu$-trees for $\nu=ENEEN$ is illustrated on the right of~\Cref{fig_altnu_lattices_ENEEN_trees}.
		
		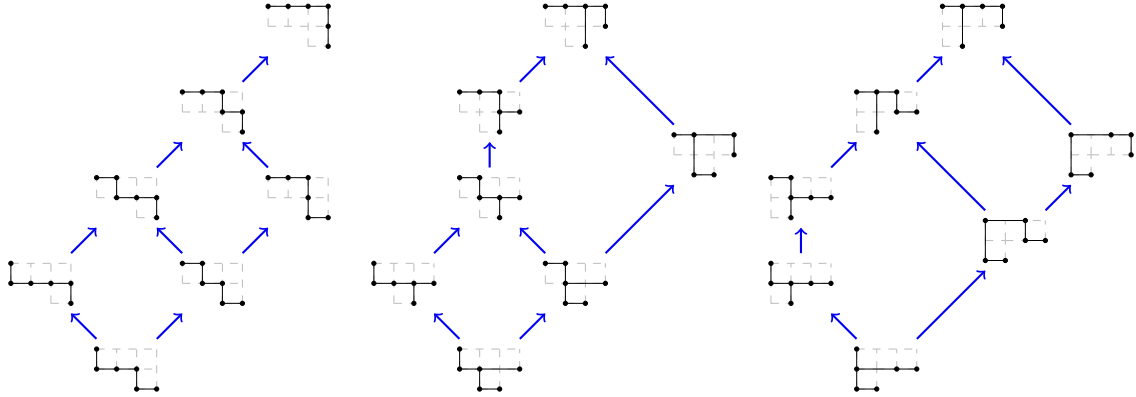
\begin{figure}[htb]
			\begin{center}
\newcommand{\customnuTree}[7]{
	\begin{tikzpicture}
		\def\dx{#5};
		\def\s{3*\dx};	
		\setcounter{n}{0}
		\setcounter{m}{0}
		\setcounter{orient}{0}
		\setcounter{curx}{0}
		\setcounter{cury}{0}
		\foreach \step in {#1}{
			\ifthenelse{\equal{\theorient}{0}}{
				\addtocounter{n}{\step}
				\setcounter{orient}{1}
			}{
				\addtocounter{m}{\step}
				\setcounter{orient}{0}
			}
		}
		\begin{pgfonlayer}{background}
			\draw[\gray,dashed](0*\dx,1*\dx) -- (0*\dx,{(2*\dx});
			\draw[\gray,dashed](1*\dx,1*\dx) -- (1*\dx,{(2*\dx});
			\draw[\gray,dashed](2*\dx,0*\dx) -- (2*\dx,{(2*\dx});		
			\draw[\gray,dashed](3*\dx,0*\dx) -- (3*\dx,{(2*\dx});
			\draw[\gray,dashed](2*\dx,0*\dx) -- (3*\dx,{(0*\dx});
			\draw[\gray,dashed](0*\dx,1*\dx) -- (3*\dx,{(1*\dx});
			\draw[\gray,dashed](0*\dx,2*\dx) -- (3*\dx,{(2*\dx});
		\end{pgfonlayer}
		\ifthenelse{\equal{#7}{}}{}{
			\coordinate(zero) at (0*\dx,0*\dx);
			\coordinate(q\thei) at(0*\dx,\then*\dx);
			\addtocounter{i}{-1}
			\foreach \k in {0,2,...,\thei}{
				\pgfmathparse{\k+1}
				\let\res\pgfmathresult
				\begin{pgfonlayer}{middle}
					\fill[#7] (zero |- q\k) -- (q\k) -- (q\res) -- (zero |- q\res) -- cycle;
				\end{pgfonlayer}
			}
		}
		\setcounter{i}{1}
		\foreach \a/\b/\c/\d in {#2}{
			\draw(\a*\dx,\b*\dx) node[fnode,draw=\c,fill=\c](p\thei){};
			\ifthenelse{\equal{\d}{d}}{
				\draw[orange!50!gray,decorate,decoration={snake,amplitude=.5,segment length=2.5}](p\thei) -- (\a*\dx,{(\then+.5)*\dx});
			}{
			}
			\stepcounter{i}
		}
		\foreach \a/\b in {#3}{
			\draw[line width=3*\dx pt](\a) -- (\b);
		}
		\foreach \a/\b/\c in {#4}{
			\draw[\c,line width=.5pt](\a*\dx,\b*\dx) circle(.25*\dx);
		}
		\foreach \a/\b/\c/\d in {#6}{
			\draw[\d](\a*\dx,\b*\dx) node[scale=.75,anchor=west]{\c};
		}
	\end{tikzpicture}
}
\begin{tikzpicture}%
	[scale=0.3,
	x={(3.8cm, 0cm)},
	y={(0cm, 3.8cm)},
	back/.style={loosely dotted, thin},
	edge/.style={color=blue!95!black, thick},
	facet/.style={fill=red!95!black,fill opacity=0.800000},
	vertex/.style={inner sep=1pt,circle,draw=green!25!black,fill=green!75!black,thick,anchor=base}]
	\scriptsize
	
	\coordinate (c120) at (0,0);
	\coordinate (c030) at (-1,1);
	\coordinate (c021) at (0,2);
	\coordinate (c012) at (1,3);
	\coordinate (c003) at (2,4);
	
	\coordinate (c111) at (1,1);
	\coordinate (c102) at (2,2);
	
	\node (n120) at (c120) {\customnuTree{0,3,2}{2/0//,3/0//,0/1//,1/1//,2/1//,0/2//}{p1/p2,p1/p5,p3/p4,p4/p5,p3/p6}{}{0.07}{}{}};
	\node (n030) at (c030) {\customnuTree{0,3,2}{3/0//,0/1//,1/1//,2/1//,3/1//,0/2//}{p1/p5,p2/p3,p3/p4,p4/p5,p2/p6}{}{0.07}{}{}};
	\node (n021) at (c021) {\customnuTree{0,3,2}{3/0//,1/1//,2/1//,3/1//,0/2//,1/2//}{p1/p4,p2/p3,p3/p4,p2/p6,p5/p6}{}{0.07}{}{}};
	\node (n012) at (c012) {\customnuTree{0,3,2}{3/0//,2/1//,3/1//,0/2//,1/2//,2/2//}{p1/p3,p2/p3,p2/p6,p4/p5,p5/p6}{}{0.07}{}{}};
	\node (n003) at (c003) {\customnuTree{0,3,2}{3/0//,3/1//,0/2//,1/2//,2/2//,3/2//}{p1/p2,p2/p6,p3/p4,p4/p5,p5/p6}{}{0.07}{}{}};

	\node (n111) at (c111) {\customnuTree{0,3,2}{2/0//,3/0//,1/1//,2/1//,0/2//,1/2//}{p1/p2,p1/p4,p3/p4,p3/p6,p5/p6}{}{0.07}{}{}};
	\node (n102) at (c102) {\customnuTree{0,3,2}{2/0//,3/0//,2/1//,0/2//,1/2//,2/2//}{p1/p2,p1/p3,p3/p6,p4/p5,p5/p6}{}{0.07}{}{}};
%

\draw[edge,->] (n120) -- (n030); 
\draw[edge,->] (n030) -- (n021); 
\draw[edge,->] (n021) -- (n012); 
\draw[edge,->] (n012) -- (n003); 

\draw[edge,->] (n120) -- (n111);
\draw[edge,->] (n111) -- (n021);
\draw[edge,->] (n111) -- (n102);

\draw[edge,->] (n102) -- (n012);
\end{tikzpicture}
\renewcommand{\customnuTree}[7]{
	\begin{tikzpicture}
		\def\dx{#5};
		\def\s{3*\dx};	
		\setcounter{n}{0}
		\setcounter{m}{0}
		\setcounter{orient}{0}
		\setcounter{curx}{0}
		\setcounter{cury}{0}
		\foreach \step in {#1}{
			\ifthenelse{\equal{\theorient}{0}}{
				\addtocounter{n}{\step}
				\setcounter{orient}{1}
			}{
				\addtocounter{m}{\step}
				\setcounter{orient}{0}
			}
		}
		\begin{pgfonlayer}{background}
			\draw[\gray,dashed](0*\dx,1*\dx) -- (0*\dx,{(2*\dx});
			\draw[\gray,dashed](1*\dx,0*\dx) -- (1*\dx,{(2*\dx});
			\draw[\gray,dashed](2*\dx,0*\dx) -- (2*\dx,{(2*\dx});		
			\draw[\gray,dashed](3*\dx,1*\dx) -- (3*\dx,{(2*\dx});
			\draw[\gray,dashed](1*\dx,0*\dx) -- (2*\dx,{(0*\dx});
			\draw[\gray,dashed](0*\dx,1*\dx) -- (3*\dx,{(1*\dx});
			\draw[\gray,dashed](0*\dx,2*\dx) -- (3*\dx,{(2*\dx});
		\end{pgfonlayer}
		\ifthenelse{\equal{#7}{}}{}{
			\coordinate(zero) at (0*\dx,0*\dx);
			\coordinate(q\thei) at(0*\dx,\then*\dx);
			\addtocounter{i}{-1}
			\foreach \k in {0,2,...,\thei}{
				\pgfmathparse{\k+1}
				\let\res\pgfmathresult
				\begin{pgfonlayer}{middle}
					\fill[#7] (zero |- q\k) -- (q\k) -- (q\res) -- (zero |- q\res) -- cycle;
				\end{pgfonlayer}
			}
		}
		\setcounter{i}{1}
		\foreach \a/\b/\c/\d in {#2}{
			\draw(\a*\dx,\b*\dx) node[fnode,draw=\c,fill=\c](p\thei){};
			\ifthenelse{\equal{\d}{d}}{
				\draw[orange!50!gray,decorate,decoration={snake,amplitude=.5,segment length=2.5}](p\thei) -- (\a*\dx,{(\then+.5)*\dx});
			}{
			}
			\stepcounter{i}
		}
		\foreach \a/\b in {#3}{
			\draw[line width=3*\dx pt](\a) -- (\b);
		}
		\foreach \a/\b/\c in {#4}{
			\draw[\c,line width=.5pt](\a*\dx,\b*\dx) circle(.25*\dx);
		}
		\foreach \a/\b/\c/\d in {#6}{
			\draw[\d](\a*\dx,\b*\dx) node[scale=.75,anchor=west]{\c};
		}
	\end{tikzpicture}
}
\begin{tikzpicture}%
	[scale=0.3,
	x={(3.8cm, 0cm)},
	y={(0cm, 3.8cm)},
	back/.style={loosely dotted, thin},
	edge/.style={color=blue!95!black, thick},
	facet/.style={fill=red!95!black,fill opacity=0.800000},
	vertex/.style={inner sep=1pt,circle,draw=green!25!black,fill=green!75!black,thick,anchor=base}]
	\scriptsize
	
	\coordinate (c120) at (0,0);
	\coordinate (c030) at (-1,1);
	\coordinate (c021) at (0,2);
	\coordinate (c012) at (0,3);
	\coordinate (c003) at (1,4);
	
	\coordinate (c111) at (1,1);
	\coordinate (c102) at (2.5,2.5);
	
	\node (n120) at (c120) {\customnuTree{0,2,1,1,1}{1/0//,2/0//,0/1//,1/1//,3/1//,0/2//}{p1/p2,p1/p4,p3/p4,p4/p5,p3/p6}{}{0.07}{}{}};
	\node (n030) at (c030) {\customnuTree{0,2,1,1,1}{2/0//,0/1//,1/1//,2/1//,3/1//,0/2//}{p1/p4,p2/p3,p3/p4,p4/p5,p2/p6}{}{0.07}{}{}};
	\node (n021) at (c021) {\customnuTree{0,2,1,1,1}{2/0//,1/1//,2/1//,3/1//,0/2//,1/2//}{p1/p3,p2/p3,p3/p4,p2/p6,p5/p6}{}{0.07}{}{}};
	\node (n012) at (c012) {\customnuTree{0,2,1,1,1}{2/0//,2/1//,3/1//,0/2//,1/2//,2/2//}{p1/p2,p2/p3,p2/p6,p4/p5,p5/p6}{}{0.07}{}{}};
	\node (n003) at (c003) {\customnuTree{0,2,1,1,1}{2/0//,3/1//,0/2//,1/2//,2/2//,3/2//}{p1/p5,p2/p6,p3/p4,p4/p5,p5/p6}{}{0.07}{}{}};
	
	\node (n111) at (c111) {\customnuTree{0,2,1,1,1}{1/0//,2/0//,1/1//,3/1//,0/2//,1/2//}{p1/p2,p1/p3,p3/p4,p3/p6,p5/p6}{}{0.07}{}{}};
	\node (n102) at (c102) {\customnuTree{0,2,1,1,1}{1/0//,2/0//,3/1//,0/2//,1/2//,3/2//}{p1/p2,p1/p5,p3/p6,p4/p5,p5/p6}{}{0.07}{}{}};
%
	
	\draw[edge,->] (n120) -- (n030); 
	\draw[edge,->] (n030) -- (n021); 
	\draw[edge,->] (n021) -- (n012); 
	\draw[edge,->] (n012) -- (n003); 
	
	\draw[edge,->] (n120) -- (n111);
	\draw[edge,->] (n111) -- (n102);
	\draw[edge,->] (n102) -- (n003);
	
	\draw[edge,->] (n111) -- (n021);
\end{tikzpicture}
				\begin{tikzpicture}%
	[scale=0.3,
	x={(3.8cm, 0cm)},
	y={(0cm, 3.8cm)},
	back/.style={loosely dotted, thin},
	edge/.style={color=blue!95!black, thick},
	facet/.style={fill=red!95!black,fill opacity=0.800000},
	vertex/.style={inner sep=1pt,circle,draw=green!25!black,fill=green!75!black,thick,anchor=base}]
	\scriptsize
	
	\coordinate (c120) at (0,0);
	\coordinate (c030) at (-1,1);
	\coordinate (c021) at (-1,2);
	\coordinate (c012) at (0,3);
	\coordinate (c003) at (1,4);
	
	\coordinate (c111) at (1.5,1.5);
	\coordinate (c102) at (2.5,2.5);
	
	\node (n120) at (c120) {\nuTree{0,1,1,2,1}{0/0//,1/0//,0/1//,2/1//,3/1//,0/2//}{p1/p2,p1/p3,p3/p4,p4/p5,p3/p6}{}{0.07}{}{}};
	\node (n030) at (c030) {\nuTree{0,1,1,2,1}{1/0//,0/1//,1/1//,2/1//,3/1//,0/2//}{p1/p3,p2/p3,p3/p4,p4/p5,p2/p6}{}{0.07}{}{}};
	\node (n021) at (c021) {\nuTree{0,1,1,2,1}{1/0//,1/1//,2/1//,3/1//,0/2//,1/2//}{p1/p2,p2/p3,p3/p4,p2/p6,p5/p6}{}{0.07}{}{}};
	\node (n012) at (c012) {\nuTree{0,1,1,2,1}{1/0//,2/1//,3/1//,0/2//,1/2//,2/2//}{p1/p5,p2/p3,p2/p6,p4/p5,p5/p6}{}{0.07}{}{}};
	\node (n003) at (c003) {\nuTree{0,1,1,2,1}{1/0//,3/1//,0/2//,1/2//,2/2//,3/2//}{p1/p4,p2/p6,p3/p4,p4/p5,p5/p6}{}{0.07}{}{}};
	
	\node (n111) at (c111) {\nuTree{0,1,1,2,1}{0/0//,1/0//,2/1//,3/1//,0/2//,2/2//}{p1/p2,p1/p5,p3/p4,p3/p6,p5/p6}{}{0.07}{}{}};
	\node (n102) at (c102) {\nuTree{0,1,1,2,1}{0/0//,1/0//,3/1//,0/2//,2/2//,3/2//}{p1/p2,p1/p4,p3/p6,p4/p5,p5/p6}{}{0.07}{}{}};

	\draw[edge,->] (n120) -- (n030); 
	\draw[edge,->] (n030) -- (n021); 
	\draw[edge,->] (n021) -- (n012); 
	\draw[edge,->] (n012) -- (n003); 
	
	\draw[edge,->] (n120) -- (n111);
	\draw[edge,->] (n111) -- (n102);
	\draw[edge,->] (n102) -- (n003);
	
	\draw[edge,->] (n111) -- (n012);
\end{tikzpicture}
			\end{center}
			\caption{Examples of rotation lattices of $(\delta,\nu)$-trees for $\nu=ENEEN$. Left: the $\nu$-Dyck lattice, for $\delta=(0,0)$. Middle: the lattice for $\delta=(1,0)$. Right: the $\nu$-Tamari lattice, for $\delta=(2,0)$.}
			\label{fig_altnu_lattices_ENEEN_trees}
		\end{figure}
		
		\begin{theorem}[{\cite{ceballos_nu_trees_2020}}]
			The $\nu$-Tamari lattice is isomorphic to the rotation poset of~$\nu$-trees:
			$$
			\Tam{\nu} \cong \TamTrees{\nu}.
			$$
			In particular, the rotation poset of~$\nu$-trees is a lattice.
		\end{theorem}
		
		A bijection between these two posets is given by the \emph{right flushing bijection} $\flush_\nu$ introduced in~\cite{ceballos_nu_trees_2020}. This bijection maps a $\nu$-path $\mu=(\mu_0,\dots,\mu_n)$ to the unique $\nu$-tree with $\mu_i+1$ nodes at height $i$. 
		This tree can be recursively obtained by adding $\mu_i+1$ nodes at height~$i$ from bottom to top, from right to left, avoiding forbidden positions. The forbidden positions are those above a node that is not the left most node in a row (these come from the initial points of the east steps in the path $\mu$). In Figure~\ref{fig_right_flushing}, the forbidden positions are the ones that belong to the wiggly lines. Note that the order of the nodes per row is reversed.
		
		The inverse $\flush_\nu^{-1}$ of the right flushing bijection is called the \emph{left flushing bijection}, and can be described similarly, adding points from left to right, from bottom to top, avoiding the forbidden position given by the wiggly lines. 
        In other words, the left flushing bijection of a $\nu$-tree is the $\nu$-path that has as many nodes per row as the tree. 
		
		\begin{figure}
			
			\begin{center}
				\centering
				\begin{minipage}[c]{.4\linewidth}
					\def\svgwidth{\linewidth}
					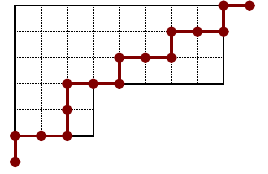
				\end{minipage}
				$\quad \overset{\text{right-flushing}}{\longrightarrow} \quad$
				\begin{minipage}[c]{.4\linewidth}
					\def\svgwidth{\linewidth}
					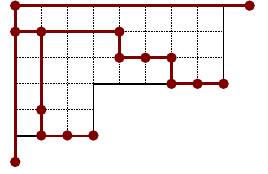
				\end{minipage}
			\end{center}
			\caption{Right flushing bijection from $\nu$-paths to $\nu$-trees.}
			\label{fig_right_flushing}
		\end{figure}
		
		\subsection{Left and right intervals in the $\nu$-Tamari lattice}
		
		The description of the $\nu$-Tamari lattice on  $\nu$-trees gives an easy description of its linear intervals.
		
		\begin{definition}
			An interval $[T, T']$ in $\TamTrees{\nu}$ is a \emph{left interval} if $ T' $ is obtained from $ T $ by applying $ \ell>0 $ rotations at the first $\ell$ nodes of a consecutive sequence $q_0,\dots, q_{\ell-1},q_\ell$ in the same row, from left to right. For example, applying two rotations at the first two nodes of the sequence $\bar p_{13},\bar  p_{12},\bar  p_{11}$ in Figure~\ref{fig_right_flushing} (right).
			It is a \emph{right interval} if $ T' $ is obtained from $ T $ by applying $ \ell $ rotations  at the first $\ell$ nodes of a consecutive sequence  $q_0,\dots, q_{\ell-1},q_\ell$ which are consecutive in the same column, from bottom to top. For example, applying two rotations at the first two nodes of the sequence $\bar p_{3}, \bar p_{4}, \bar p_{12}$ in Figure~\ref{fig_right_flushing} (right).
		\end{definition}
		
		\begin{proposition}\label{prop_left_right_vTam}
			The left and right intervals in the previous definition are linear intervals of length~$\ell$. Moreover,  all non trivial linear intervals in the rotation lattice on $\nu$-trees are either left or right intervals.
		\end{proposition}

		\begin{proof}
			A left (resp. right) interval $ [T,T'] $ is indeed linear because there is a unique maximal chain from $ T $ to $ T' $ and $ \ell $ is its length.
			Indeed, each element different from $ T' $ in the interval has only one upper cover that is below $ T' $. 
			
			\medskip
			The converse is proven by induction, similarly as the simpler case of the Dyck lattice in \Cref{prop:linear_in_Dyck}.
			
			\noindent Intervals of length $ 1 $ are all covering relations and thus both left and right intervals.
			
			\noindent An interval which contains a maximal chain $ T_0 \lessdot T_1 \lessdot T_2 $ can be:
			\begin{itemize}
				\item a left interval when $ T_1 \lessdot T_2 $ is the rotation in the same row immediately on the right of $ T_0 \lessdot T_1 $,
				\item a right interval when $ T_1 \lessdot T_2 $ is the rotation in the same column immediately above $ T_0 \lessdot T_1 $,
				\item a pentagon when  $ T_1 \lessdot T_2 $ is the rotation in the same column immediately under $ T_0 \lessdot T_1 $, 
				\item a square otherwise.
			\end{itemize}
			\noindent In particular, this proves that all linear intervals of length $ 2 $ are exactly left or right intervals.
			
			\medskip
			\noindent A linear interval $ [T,T'] $ of length $ k \geq 3 $ must contain a linear interval $ [T,T''] $ of length $ k-1 $.
			By induction, $ [T,T''] $ must be either a left interval and in this case $ [T,T'] $ is a left interval as well or a right interval and in this case $ [T,T'] $ is also a right interval.
		\end{proof}

		\begin{remark} \label{rem:nulinear}
			The left flushing of a left interval on the rotation lattice of $\nu$-trees produces a left interval~$ [P, Q] $ of $\nu$-paths in $ \Tam{\nu} $, where $ P $ is of the form $ A E^k B C $ with $ B $ some $\nu$-excursion and $ Q $ is of the form $ A B E^k C $.
			In other words, $ P $ is a $\nu$-path with a valley preceded by $ k $ east steps.
			
			The left flushing of a right interval on the rotation lattice of $\nu$-trees produces a right interval $ [P, Q] $ of $\nu$-paths in $ \Tam{\nu} $, where $ P $ is of the form $ A E B_1 \dots B_k C $ with $ B_1, \dots, B_k$ being $ k $ consecutive $\nu$-excursions, and $ Q $ is of the form $ A B_1 \dots B_k E C$.
		\end{remark}
		
		\section{The alt $\nu$-Tamari lattice}\label{sec_alt_nu_Tamari}
		
		Given a fixed path $\nu$, the $\nu$-Dyck lattice and the $\nu$-Tamari lattice are two posets defined on~$\nu$-paths with quite similar covering relations.
		In both cases, a covering relation consists of swapping the east step of a valley with a subpath that follows it.
		We can in fact define a whole family of posets that are described in a similar way, and we call them the alt~$\nu$-Tamari posets.
		The term ``alt'' stands for ``altitude'', a notion that we use in order to define them.
		We prove that the resulting posets are lattices and study their linear intervals.
		
		\subsection{On $\nu$-paths}
		Let $\nu = (\nu_0, \dots, \nu_n)$ be a fixed path.
		We say that $\delta = (\delta_1, \dots, \delta_n) \in \mathbb{N}^n$ is an increment vector with respect to $\nu$ if $\delta_i \leq \nu_i$ for all $ 1 \leq i \leq n $.

		Similarly as the $\nu$-altitude, we introduce a notion of $\delta$-altitude. 
		For a lattice point $ p $ on a $\nu$-path~$\mu$, define its \emph{$\delta$-altitude} $ \alt_\delta(p) $ as follows. 
		We set the $\delta$-altitude of the initial lattice point of $\mu$ to be equal to zero,
		and declare that the $ i $-th north step of $\mu$ increases the $\delta$-altitude by $\delta_i$ and an east step decreases the $\delta$-altitude by $ 1 $.
		
		Given a valley $ EN $ of $\mu$, let $ p $ be the lattice point between the east and north steps.
		Let $ q $ be the next lattice point of $ \mu $ such that $ \alt_\delta(q) =  \alt_\delta(p) $, and $ \mu_{[p,q]} $ be the subpath of $\mu$ that starts at $ p $ and ends at $ q $.
		Let $\mu'$ be the path obtained from $\mu$ by switching $ \mu_{[p,q]} $ with the east step $E$ that precedes it.
		The $\delta$-rotation of~$\mu$ at the valley $ p $ is defined to be $ \mu \drot{\delta}  \mu' $. An example is illustrated in Figure~\ref{fig_delta_nu_path_rotation}.

		\begin{figure}[htb]
			\centering
		\begin{center}
			\centering
			\begin{minipage}[c]{.4\linewidth}
				\hfill
				\def\svgwidth{.8\linewidth}
				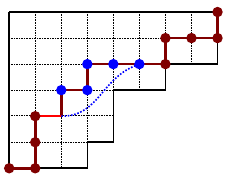
			\end{minipage}
			$\quad \lessdot_{\delta} \quad$
			\begin{minipage}[c]{.4\linewidth}
				\def\svgwidth{.8\linewidth}
				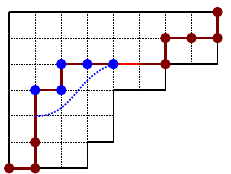
				\hfill
			\end{minipage}
		\end{center}
			\caption{The $\delta$-rotation operation of a $\nu$-path for $\delta=(0,0,2,1,0,0)$. Each node is labelled with its $\delta$-altitude.}    
			\label{fig_delta_nu_path_rotation}
		\end{figure}
		
		Note that a $\delta$-rotation increases the number of boxes below the path, and therefore its reflexive transitive closure induces a poset structure on the set of $\nu$-paths. 
		
		\begin{definition}
			Let $\delta$ be an increment vector with respect to $\nu$.
			The \emph{alt $\nu$-Tamari poset} $ \altTam{\nu}{\delta} $ is the reflexive transitive closure of $\delta$-rotations on the set of $\nu$-paths.
		\end{definition}
		
		The three examples of the $\nu$-Tamari poset for $\nu=ENEEN=(1,2,0)$ are illustrated on~\Cref{fig_altnu_lattices_ENEEN_paths}. 
		
		\begin{remark}
			For a fixed path $\nu$, there are two extreme choices for the increment vector $\delta$.
			If $\delta_i = \nu_i$ for all $ 1 \leq i \leq n $, the alt $\nu$-Tamari lattice coincides with the $\nu$-Tamari lattice.
			If~$\delta_i = 0$ for all $ 1 \leq i \leq n $, the alt $\nu$-Tamari lattice coincides with the $\nu$-Dyck lattice. 
			We denote these two cases by~$\delta^{\max}$ and~$\delta^{\min}$, respectively.
		\end{remark}
		
		Another approach to define the alt $\nu$-Tamari poset is to introduce the notion of \emph{$\delta$-elevation} of a subpath as the difference of the $\delta$-altitude between its ending point and its starting point.    
		We thus write $ \elev_\delta(E) = -1 $ for an east step $E$ and $\elev_\delta(N_i) = \delta_i$ if $ N_i $ is the $ i $-th north step of a $\nu$-path $\mu$.
		For any subpath $ A $ of $\mu$, we then have $\elev_\delta(A) = \sum_{a \in A} \elev_\delta(a)$ as the sum of the $\delta$-elevation of the steps of $A$.
		
		The \emph{$\delta$-excursion} of a north step $ N $ of a $\nu$-path $\mu$ is defined as the shortest subpath~$ A $ of $ \mu $ that starts with this $ N $ and such that $ \elev_\delta(A) = 0 $.
		It follows from the definition of the $\delta$-excursion that exchanging the east step $ E $ of a valley with the $\delta$-excursion that follows it is exactly a covering relation in $ \altTam{\nu}{\delta} $.
		
		\begin{remark}
			Note that for $\delta = \delta^{\max}$, the $\delta$-altitude is the $\nu$-altitude shifted by $-\nu_0$, but the $\delta$-elevation and the $\nu$-elevation are equal.
		\end{remark}
		
		For a general increment vector $\delta$ with respect to $\nu$, it is not a priori clear that $ \altTam{\nu}{\delta} $ is a lattice. This is a consequence of the following proposition.
		
		\begin{proposition}\label{prop_delta_poset_nuprime}
			Let $ \check\nu_0 = \sum_{i = 0}^{n} \nu_i - \sum_{i = 1}^{n} \delta_i$ with $\delta_i \leq \nu_i$.
			Then $\check\nu = (\check\nu_0,\check\nu_1,\dots, \check\nu_n) = (\check\nu_0, \delta_1, \dots, \delta_n)$ is a path below $\nu$ whose endpoints are the same as $\nu$. Moreover, the following properties hold:
			\begin{enumerate}
				\item $\delta$-rotations of a~$\nu$-path $ \mu $ coincide with $\check\nu$-rotations of $ \mu $. \label{prop_delta_poset_nuprime_item1}
				\item The alt $\nu$-Tamari poset $ \altTam{\nu}{\delta} $ is the restriction of $ \Tam{\check\nu} $ to the subset of paths weakly above~$\nu$.\label{prop_delta_poset_nuprime_item2}
				\item The covering relations of $ \altTam{\nu}{\delta} $ are exactly the $\delta$-rotations. \label{prop_delta_poset_nuprime_item3}
				\item The alt $\nu$-Tamari poset $ \altTam{\nu}{\delta} $ is the interval $[\nu,1^\nu]$ in $ \Tam{\check\nu} $. \label{prop_delta_poset_nuprime_item4}
			\end{enumerate}
			Here, $1^\nu=N^nE^{m}$ denotes the top path above $\nu$ and $\check\nu$, where $m=\nu_0+\dots+\nu_n=\check\nu_0+\dots+\check\nu_n$. 
		\end{proposition}

		\begin{proof}
			The first part of the statement follows from  
			$
			\sum_{i=0}^j \check\nu_i - \sum_{i=0}^j \nu_i = \sum_{i=j+1}^n (\nu_i-\delta_i) \geq 0, 
			$
			where equality holds for $j=n$.
			
			Note that a $\nu$-path $ \mu $ is also a $\check\nu$-path, and for a subpath $A$ of $\mu$ we have $\elev_\delta(A)=\elev_{\check\nu}(A)$. Since the $\delta$-rotations (resp. $\check\nu$-rotations) are determined by the $\delta$-elevation (resp. $\check\nu$-elevation), then Item~\eqref{prop_delta_poset_nuprime_item1} follows. Items~\eqref{prop_delta_poset_nuprime_item2} and~\eqref{prop_delta_poset_nuprime_item3} follow from Item~\eqref{prop_delta_poset_nuprime_item1}.
			
			For Item~\eqref{prop_delta_poset_nuprime_item4} we need to show that 
			the restriction of $ \Tam{\check\nu} $ to the subset of paths weakly above~$\nu$ is the interval $[\nu,1^\nu]$ in $ \Tam{\check\nu} $.
			In other words, we need to show that every $\nu$-path $\mu$ satisfies $\nu \leq_{\Tam{\check\nu}} \mu$. 
			Note that this property does not hold for an arbitrary path $\check\nu$ below $\nu$, but for our particular choice this is equivalent to show that~$\nu \leq_{\altTam{\nu}{\delta}} \mu$ (by Item~\eqref{prop_delta_poset_nuprime_item1}). This holds because we can reach any $\nu$-path~$\mu$ by applying a sequence of $\delta$-rotations: add the boxes between $\nu$ and $\mu$ one at a time from bottom to top, from right to left. Each of these steps corresponds to a $\delta$-rotation because $\delta_i\leq \nu_i$. 
		\end{proof}
		
		\begin{corollary}\label{cor_alt_Tamari_is_lattice}
			The alt $\nu$-Tamari poset is a lattice.
		\end{corollary}
		
		\begin{proof}
			By~\Cref{prop_delta_poset_nuprime}~\eqref{prop_delta_poset_nuprime_item4}, the alt $\nu$-Tamari poset~$ \altTam{\nu}{\delta} $ is isomorphic to the interval $ [\nu,1^\nu] $ in~$ \Tam{\check\nu} $.
			Since an interval in a lattice is also a lattice, we deduce that $ \altTam{\nu}{\delta} $ is a lattice.
		\end{proof}
	
		\begin{remark}
			If we chose any other path $\check{\nu}$ weakly below $\nu$ that does not satisfy $\check{\nu}_i \leq \nu_i$, for all $i > 0$, then the restriction of $\Tam{\check{\nu}}$ to the subset of $\nu$-paths is not a lattice. It is an upper set but it has several minimal elements. The results on the number of linear intervals that are presented in the rest of the paper do not hold either with this weaker condition (see \Cref{rem:wrongcase}).
		\end{remark}

		
		We do not use the following proposition in this paper, but it is an interesting property that we would like to highlight. 
		
		\begin{proposition} \label{prop_comparison_extension}
			Let $\delta$ and $ \delta'$ be two increment vectors with respect to $\nu$ such that $\delta_i \leq \delta_i'$ for all~$i$.
			If $ P < Q $ in $ \altTam{\nu}{\delta'}$, then $ P < Q $ in $ \altTam{\nu}{\delta} $.
		\end{proposition}
		
		In other words, whenever $\delta \leq \delta'$, the poset $ \altTam{\nu}{\delta} $ is an \emph{extension} of the poset $ \altTam{\nu}{\delta'} $, meaning that it can be obtained from $ \altTam{\nu}{\delta'} $ by adding some relations.
		
		\begin{proof}
			It is sufficient to prove that the result is true for covering relations in $ \altTam{\nu}{\delta'}$, namely that if~$ P \drot{\delta'} Q$, then $ P < Q $ in $ \altTam{\nu}{\delta} $.
			
			We can write $ P = A E B C $ and $ Q = A B E C $ for some $\delta'$-excursion $ B $.
			Note that for any north step~$ N $ in $ B $, the $\delta'$ excursion of $ N $ is a subword of $ B $ and that the $\delta$-excursion of any north step is a prefix of its $\delta'$-excursion since we have $ \delta \leq \delta' $.
			Thus, we can can build a chain of $\delta$-rotations from $ P $ to $ Q $ by exchanging this east step $ E $ with either the next $\delta$-excursions if it is followed by a north step or with the east step that follows it otherwise, which does not change the path.
		\end{proof}
		
		
			
		
		
		
		\subsection{On $(\delta,\nu)$-trees}
		
		The alt $\nu$-Tamari lattice $ \altTam{\nu}{\delta} $ is the interval $ [\nu,1^\nu] $ in $ \Tam{\check\nu} $. So, it can be described as the rotation lattice of $\check\nu$-trees that are above the $\check\nu$-tree $T_\nu$ corresponding to $\nu$ in $\Tam{\check\nu}$. These trees can be described as maximal collections of pairwise compatible elements in a shape $F_{\delta,\nu}$ which we will now describe.
		This point of view is useful to show that all alt $\nu$-Tamari lattices have the same number of linear intervals of any length. 
		
		Let $\delta, \nu$ and $\check\nu$ as in Proposition~\ref{prop_delta_poset_nuprime}. 
		Let $F_{\check\nu}$ be the Ferrers diagram that lies weakly above $\check\nu$. 
		We consider the lattice path $\hat\nu$  that starts at the lowest right corner of $F_{\check\nu}$ (the point with coordinates $(\check\nu_0,0)$) which consists of the sequence of west and north steps 
		\begin{align}\label{eq_nu_hat}
			W^{\nu_0} N W^{\gamma_1} N W^{\gamma_2} \dots N W^{\gamma_n},
			\qquad \text{ for } \gamma_i=\nu_i-\delta_i.    
		\end{align}
		We define $F_{\delta,\nu}$ to be the subset of $F_{\check\nu}$ consisting of the boxes that are weakly above $\hat\nu$, and denote by~$L_{\delta,\nu}$ its set of lattice points. 
		A \emph{$(\delta,\nu)$-tree} is a maximal collection of pairwise $\check\nu$-compatible elements in $L_{\delta,\nu}$. 
		An example is illustrated on the right of Figure~\ref{fig_Ferrers_deltanu_tree}.
		
		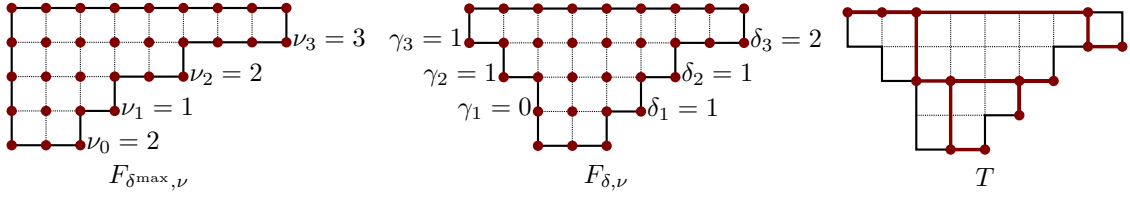
\begin{figure}[htb]
			\centering
				\begin{center}
				\centering
					\begin{minipage}[c]{.27\linewidth}
						\centering
						\def\svgwidth{.9\linewidth}
\begingroup%
  \makeatletter%
  \providecommand\color[2][]{%
    \errmessage{(Inkscape) Color is used for the text in Inkscape, but the package 'color.sty' is not loaded}%
    \renewcommand\color[2][]{}%
  }%
  \providecommand\transparent[1]{%
    \errmessage{(Inkscape) Transparency is used (non-zero) for the text in Inkscape, but the package 'transparent.sty' is not loaded}%
    \renewcommand\transparent[1]{}%
  }%
  \providecommand\rotatebox[2]{#2}%
  \newcommand*\fsize{\dimexpr\f@size pt\relax}%
  \newcommand*\lineheight[1]{\fontsize{\fsize}{#1\fsize}\selectfont}%
  \ifx\svgwidth\undefined%
    \setlength{\unitlength}{62.25001658bp}%
    \ifx\svgscale\undefined%
      \relax%
    \else%
      \setlength{\unitlength}{\unitlength * \real{\svgscale}}%
    \fi%
  \else%
    \setlength{\unitlength}{\svgwidth}%
  \fi%
  \global\let\svgwidth\undefined%
  \global\let\svgscale\undefined%
  \makeatother%
  \begin{picture}(1,0.52200023)%
    \lineheight{1}%
    \setlength\tabcolsep{0pt}%
    \put(0,0){\includegraphics[width=\unitlength,page=1]{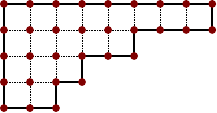}}%
    \put(1.00303597,0.3688675){\color[rgb]{0,0,0}\makebox(0,0)[lt]{\lineheight{1.25}\smash{\begin{tabular}[t]{l}$\nu_3=3$\end{tabular}}}}%
    \put(0.64158994,0.24838514){\color[rgb]{0,0,0}\makebox(0,0)[lt]{\lineheight{1.25}\smash{\begin{tabular}[t]{l}$\nu_2=2$\end{tabular}}}}%
    \put(0.40062626,0.12790348){\color[rgb]{0,0,0}\makebox(0,0)[lt]{\lineheight{1.25}\smash{\begin{tabular}[t]{l}$\nu_1=1$\end{tabular}}}}%
    \put(0.28014425,0.00742181){\color[rgb]{0,0,0}\makebox(0,0)[lt]{\lineheight{1.25}\smash{\begin{tabular}[t]{l}$\nu_0=2$\end{tabular}}}}%
  \end{picture}%
\endgroup%

						$F_{\delta^{\max},\nu}$
					\end{minipage}
				\qquad \qquad \quad
					\begin{minipage}[c]{.27\linewidth}
						\centering
						\def\svgwidth{.9\linewidth}
\begingroup%
  \makeatletter%
  \providecommand\color[2][]{%
    \errmessage{(Inkscape) Color is used for the text in Inkscape, but the package 'color.sty' is not loaded}%
    \renewcommand\color[2][]{}%
  }%
  \providecommand\transparent[1]{%
    \errmessage{(Inkscape) Transparency is used (non-zero) for the text in Inkscape, but the package 'transparent.sty' is not loaded}%
    \renewcommand\transparent[1]{}%
  }%
  \providecommand\rotatebox[2]{#2}%
  \newcommand*\fsize{\dimexpr\f@size pt\relax}%
  \newcommand*\lineheight[1]{\fontsize{\fsize}{#1\fsize}\selectfont}%
  \ifx\svgwidth\undefined%
    \setlength{\unitlength}{62.25001658bp}%
    \ifx\svgscale\undefined%
      \relax%
    \else%
      \setlength{\unitlength}{\unitlength * \real{\svgscale}}%
    \fi%
  \else%
    \setlength{\unitlength}{\svgwidth}%
  \fi%
  \global\let\svgwidth\undefined%
  \global\let\svgscale\undefined%
  \makeatother%
  \begin{picture}(1,0.51807191)%
    \lineheight{1}%
    \setlength\tabcolsep{0pt}%
    \put(0,0){\includegraphics[width=\unitlength,page=1]{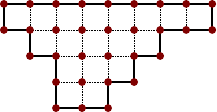}}%
    \put(1.00303632,0.36491556){\color[rgb]{0,0,0}\makebox(0,0)[lt]{\lineheight{1.25}\smash{\begin{tabular}[t]{l}$\delta_3=2$\end{tabular}}}}%
    \put(0.7620723,0.2444339){\color[rgb]{0,0,0}\makebox(0,0)[lt]{\lineheight{1.25}\smash{\begin{tabular}[t]{l}$\delta_2=1$\end{tabular}}}}%
    \put(0.64159029,0.12395223){\color[rgb]{0,0,0}\makebox(0,0)[lt]{\lineheight{1.25}\smash{\begin{tabular}[t]{l}$\delta_1=1$\end{tabular}}}}%
    \put(0.23763899,0.12604784){\color[rgb]{0,0,0}\makebox(0,0)[rt]{\lineheight{1.25}\smash{\begin{tabular}[t]{r}$\gamma_1=0$\end{tabular}}}}%
    \put(0.11715698,0.24652951){\color[rgb]{0,0,0}\makebox(0,0)[rt]{\lineheight{1.25}\smash{\begin{tabular}[t]{r}$\gamma_2=1$\end{tabular}}}}%
    \put(-0.00332469,0.36701187){\color[rgb]{0,0,0}\makebox(0,0)[rt]{\lineheight{1.25}\smash{\begin{tabular}[t]{r}$\gamma_3=1$\end{tabular}}}}%
  \end{picture}%
\endgroup%

						$F_{\delta,\nu}$
					\end{minipage}
				\qquad
					\begin{minipage}[c]{.27\linewidth}
						\centering
						\def\svgwidth{.9\linewidth}
						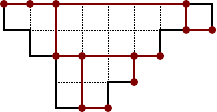
						
						$T$
					\end{minipage}
				\end{center}
			\caption{Left: The Ferrers diagram $F_{\delta^{\max},\nu}$ and its corresponding lattice points~$L_{\delta^{\max},\nu}$ for $\nu=EENENEENEEEN=(2,1,2,3,0)$ and $\delta^{\max}=(1,2,3,0)$.
				Middle: The Ferrers diagram $F_{\delta,\nu}$ and its corresponding lattice points $L_{\delta,\nu}$ for the same~$\nu$ and $\delta=(1,1,2,0)$; the path $\check \nu = EEEENENENEEN$ and $\hat \nu=WWNNWNWN$.
				Right: a $(\delta,\nu)$-tree. 
			}
			\label{fig_Ferrers_deltanu_tree}
		\end{figure}
		
		\begin{lemma} \label{lem_check_nu_tree}
			The $(\delta,\nu)$-trees are exactly the $\check{\nu}$-trees that are contained in~$L_{\delta,\nu}$.
		\end{lemma}
		
		\begin{proof}
			A $\check{\nu}$-tree that is contained in~$L_{\delta,\nu}$ is automatically a $(\delta,\nu)$-tree by definition. So, we just need to check that $(\delta,\nu)$-trees are $\check{\nu}$-trees.
			By definition, a $(\delta,\nu)$-tree is a maximal collection of pairwise $\check\nu$-compatible elements in~$L_{\delta,\nu}$. 
			As $L_{\delta,\nu} \subseteq L_{\check{\nu}}$, we want to prove that the maximality in $ L_{\delta,\nu} $ implies the maximality in $ L_{\check\nu} $.
			
			Recall that the paths $\nu$ and $\check \nu$ have the same starting point $(0,0)$ and the same ending point $(m,n)$, and that every $\nu$-tree and every $\check \nu$-tree has exactly $m+n+1$ nodes (equal to the number of lattice points in $\nu$ and $\check \nu$). Furthermore, the shape $F_{\delta,\nu}$ fits in the $m\times n$ box with the top corners being $(0,n)$ and $(m,n)$. In our example in~\Cref{fig:BijectionLeftIntervals}, $m=11$ and $n=7$. The $\nu$-tree and the $(\delta,\nu)$-tree shown in this figure both have $m+n+1=19$ nodes. We want to show that every $(\delta,\nu)$-tree has exactly~$m+n+1$ elements. 
			
			Let  $ T $ be a $(\delta,\nu)$-tree. Label its elements $p_0,p_1,\dots ,p_r$ from bottom to top, from right to left. We will show that $r=m+n$, which implies that $T$ has $m+n+1$ elements as desired.
			
			Let us reconstruct $T$ recursively, by adding the elements $p_0,p_1,\dots,p_r$ one at a time in order. Note that if $p_i$ is not the left most element in its row, then all the lattice points above $p_i$ are forbidden in the next steps, because they are incompatible with an element $p_j\in T$ that is to the left of $p_i$ in the same row. 
			
			Now, when we add an element $p_j$ in the process of reconstructing $T$, then $p_j$ is necessarily located at the right most position of its row that is not forbidden by any element before. Otherwise, let $p_j\in T$ be the node with smallest label that does not satisfy that property, and let $q$ be the right most lattice point in the same row that is not forbidden by any element $p_i$ with $i<j$. In particular, $q$ is on the right of~$p_j$ by assumption, and $q$ is compatible with every $p_i$ with $i<j$. Moreover, $q$ is also compatible with $p_k\in T$ with $k>j$, otherwise $p_k,p_j$ would be incompatible. So, we can add the element $q$ to $T$, creating a new compatible set, contradicting the maximality of $T$.   
			
			Furthermore, the following relation holds,
			\[
			j = \operatorname{forb}(p_j) + \operatorname{height}(p_j)
			\]
			where $\operatorname{height}(p_j)$ is the height of $p_j\in T$ and $\operatorname{forb}(p_j)$ is the number of $p_i\in T$ with $i<j$ such that $p_i$ is not the left most node of $T$ in its row. 
			That is, $\operatorname{forb}(p_j)$ is the number of nodes before $p_j$ that forbid the positions above them. This formula is clear because the $j$ nodes $p_0,\dots,p_{j-1}$ appearing before $p_j$ either forbid positions above them (not the left most node of their row) or increase the height by one (the left most node of their row).
			
			If we apply this formula to the last node $p_r$ and assume that $r<m+n$, then 
			\[
			\operatorname{forb}(p_r) + \operatorname{height}(p_r) < m+n.
			\]
			We can assume that $\operatorname{height}(p_r)=n$ (maximum possible height), otherwise we could add the top left corner of $F_{\delta,\nu}$ to $T$, creating a bigger compatible set and contradicting the maximality $T$. 
			This implies that $\operatorname{forb}(p_r)<m$, which means that on the top row there are still some lattice points that are not forbidden. Adding one of these points contradicts the maximality of $T$. 
			As a consequence, we have proven that $r=m+n$ as desired.
		\end{proof}

		We define the \emph{$(\delta,\nu)$-right flushing} $\flush_{\delta,\nu}$ as the restriction of the right flushing bijection $\flush_{\check \nu}$ (with respect to $\check \nu$) to set of $\nu$-paths (thought as the subset of $\check \nu$-paths that are above $\nu$).  
		
		\begin{proposition}\label{prop_delta_nu_flushing}
			The map $\flush_{\delta,\nu}$ is a bijection between the set of $\nu$-paths and the set of $(\delta,\nu)$-trees. Moreover, 
			two $\nu$-paths are related by a $\delta$-rotation 
			$ \mu \drot{\delta} \mu' $ 
			if and only if the corresponding trees 
			are related by a ${\check \nu}$-rotation
			$T \drot{\check \nu} T'$.  
		\end{proposition}
		
		\begin{proof} 
			By~\Cref{prop_delta_poset_nuprime}, $\delta$-rotations of a~$\nu$-path $ \mu $ coincide with $\check\nu$-rotations of $ \mu $, and
			the right flushing bijection $\flush_{\check\nu}$ transforms $\check \nu$-rotations on paths to $\check\nu$-rotations on the corresponding trees. 
			Therefore,
			the second part of the proposition is clear. 
			It remains to show that $\mu$ is a $\nu$-path if and only if $\flush_{\check\nu}(\mu)$ is a $(\delta,\nu)$-tree, or equivalently a $\check \nu$-tree that is contained in~$L_{\delta,\nu}$. 
			
			We start by proving the forward direction. 
			First, note that the image of the bottom path $T_\nu=\flush_{\check\nu}(\nu)$ is contained in~$L_{\delta,\nu}$. 
			More precisely, the shape $F_{\delta,\nu}$ has $\nu_k$ east steps on its boundary at height $k$. For $k>0$, some of these east steps (exactly $\nu_k-\delta_k$) are on the left boundary, and some (exactly~$\delta_k$) are on the right boundary. The $\nu_k+1$ nodes of $T_\nu$ at height $k$ consist of the $\nu_k-\delta_k+1$ points on the left boundary, and the $\delta_k$ end points of the east steps of the right boundary. At height~$k=0$, the~$\nu_0+1$ nodes of $T$ are all the lattice points at the bottom of $F_{\delta,\nu}$.  
			This shows that $T_\nu$ is contained in~$L_{\delta,\nu}$.  
			
			Now, every $\nu$-path $\mu$ can be obtained by applying a sequence of $\check \nu$-rotations to the bottom path~$\nu$. Its image $\flush_{\check\nu}(\mu)$ is a $\check \nu$-tree that can be obtained by applying the corresponding sequence of $\check \nu$-rotations to the tree $T_\nu$. Since $T_\nu$ is contained in~$L_{\delta,\nu}$ and such rotations preserve this property, then $\flush_{\check\nu}(\mu)$ is also contained in~$L_{\delta,\nu}$. 
			This finishes the proof of the forward direction. 
			
			The backward direction is equivalent to the following statement: if $T$ is a $\check \nu$-tree contained in~$L_{\delta,\nu}$ then $\mu=\flush_{\check\nu}^{-1}(T)$ is weakly above $\nu$. 
			This is equivalent to show that the number of nodes in $T$ at heights less than or equal to $k$ is at most $\nu_0+\dots +\nu_k+k+1$.
			
			For $0\leq k \leq n$, let $L_k$ (resp. $F_k$) be the the restriction of $L_{\delta,\nu}$ (resp. $F_{\delta,\nu}$) to the points with height less than or equal to $k$. The width of $F_k$ is equal to $\nu_0+\dots+\nu_k$. The maximal number of compatible lattice points inside $L_k$ is equal to $\nu_0+\dots+\nu_k+k+1$. 
			The restriction of $T$ to $L_k$ is a compatible set (not necessarily maximal). The result follows. 
		\end{proof}

		The $(\delta,\nu)$-right flushing bijection from~\mbox{$\nu$-paths} to $(\delta,\nu)$-trees is described in exactly the same way as in Section~\ref{sec_nu_trees}: 
		we recursively add $\mu_i+1$ nodes to the tree inside the shape~$F_{\delta,\nu}$ from right to left, from bottom to top, while avoiding the forbidden positions above a node which is not the left most node in a row.  
		Figure~\ref{fig:BijectionLeftIntervals} shows the image of the path~\mbox{$\mu=(1,0,1,1,3,2,1,2)$} for $\delta^{max}=(1,0,2,2,0,3,0)$ (left) and for $\delta=(0,0,1,2,0,1,0)$ (right), where the base path is $\nu=(3,1,0,2,2,0,3,0)$.
		
		\begin{definition}
			The \emph{rotation poset of $(\delta,\nu)$-trees} $\altTamTrees{\nu}{\delta}$ is the reflexive transitive closure of \mbox{$\check\nu$-rotations} on~$(\delta,\nu)$-trees.   
		\end{definition}
		
		The three examples of the rotation poset of $(\delta,\nu)$-trees for $\nu=ENEEN=(1,2,0)$ are illustrated on~\Cref{fig_altnu_lattices_ENEEN_trees}. 
		
		\begin{theorem}
			The alt $\nu$-Tamari lattice is isomorphic to the rotation poset of $(\delta,\nu)$-trees: $$\altTam{\nu}{\delta}\cong \altTamTrees{\nu}{\delta}.$$  
			In particular, the rotation poset of $(\delta,\nu)$-trees is a lattice.  
		\end{theorem}
		
		\begin{proof}
			The alt $\nu$-Tamari lattice is the poset on $\nu$-paths whose covering relations are given by \mbox{$\delta$-rotations}.
			The rotation poset of $(\delta,\nu)$-trees is poset on $(\delta,\nu)$-trees whose covering relations are \mbox{$\check\nu$-rotations}.
			The result is then a consequence of~\Cref{prop_delta_nu_flushing}.
			The lattice property was proven for the alt $\nu$-Tamari lattice in~\Cref{cor_alt_Tamari_is_lattice}.
		\end{proof}

		\section{Left and right intervals in the alt $\nu$-Tamari lattice}
		Since $ \altTam{\nu}{\delta} $ is an interval in $\Tam{\check\nu}$, its linear intervals are linear intervals in $\Tam{\check\nu}$. 
		In terms of trees, this gives the following simple characterization.
		
		\begin{definition}
			An interval $[T, T']$ in $\altTamTrees{\nu}{\delta}$ is a \emph{left interval} (resp. \emph{right interval}) if $[T, T']$ is a left interval (resp. right interval) in $\TamTrees{\check \nu}$. 
		\end{definition}

		\begin{proposition} \label{prop_linear_intervals_altnu}
			The non trivial linear intervals in $\altTamTrees{\nu}{\delta}$ are 
			either left or right intervals.
		\end{proposition}
		
		\begin{proof}
			This is a direct consequence of~\Cref{prop_left_right_vTam}. 
		\end{proof}
		
		In this section, we aim to characterize the left and right intervals in terms of certain row and (reduced) column vectors associated to the~$(\delta,\nu)$-trees. This will be used in~\Cref{sec_bijections_linearintervals}, to show that the number of linear intervals in the alt $\nu$-Tamari lattice~$\altTam{\nu}{\delta}$ is independent of the choice of $\delta$.   
		
		\subsection{Row vectors and left intervals}
		The \emph{row vector} of a $(\delta,\nu)$-tree $T$ is the vector 
		$$r(T)=(r_0,\dots,r_n),$$ 
		where $r_i+1$ is the number of nodes of $T$ at height $i$. 
		
		\begin{proposition}\label{prop_row_vectors_characterization}
			A $(\delta,\nu)$-tree $T$ is completely characterized by its row vector. Moreover, $(r_0,\dots,r_n)$ is the row vector of some $(\delta,\nu)$-tree if and only if
			
			\begin{enumerate}
				\item $r_i\geq 0$ for all $i$,  \label{prop_row_vectors_characterization_item1}
				\item $\sum_{i=0}^j r_i \leq \sum_{i=0}^j \nu_i 
				$ for all $j$, and \label{prop_row_vectors_characterization_item2}
				\item $\sum_{i=0}^n r_i =  \sum_{i=0}^n \nu_i$.  \label{prop_row_vectors_characterization_item3}
			\end{enumerate}
		\end{proposition}
		\begin{proof}
			By~\Cref{prop_delta_nu_flushing}, the map $\flush_{\delta,\nu}$ is a bijection between the set of $\nu$-paths and the set of~$(\delta,\nu)$-trees. 
			Moreover, this map preserves the number of points at each height, and therefore the row vector. 
			Since $\nu$-paths are characterized by their row vectors, then $(\delta,\nu)$-trees are characterized by their row vectors as well. 
			
			Furthermore, via the map $\flush_{\delta,\nu}$, characterizing the row vectors of $(\delta,\nu)$-trees is equivalent to characterizing the row vectors of $\nu$-paths. 
			Condition~\eqref{prop_row_vectors_characterization_item1} just says that every $\nu$-path $\mu$ has at least one lattice point at each height.
			Condition~\eqref{prop_row_vectors_characterization_item2} says that $\mu$ is weakly above $\nu$, and Condition~\eqref{prop_row_vectors_characterization_item3} says that~$\mu$ and~$\nu$ have the same ending points. 
		\end{proof}
		
		Given a $(\delta,\nu)$-tree $T$, we say that an \emph{ordered} set 
		$L=\{p,q_0,q_1,\dots,q_\ell \}\subseteq T$ is a \emph{horizontal L} of $T$ if $L$ is the restriction of $T$ to a rectangle $R$ of the grid, such that $p$ is the top-left corner of $R$, and $q_0,q_1,\dots , q_\ell$ appear in this order on the bottom side of $R$ with $q_0$ being its left-bottom corner and  $q_\ell$ its right-bottom corner. 
		Note that no other elements of $T$ belong to $R$. We say that the length of $L$ is equal to $\ell$. 
		We denote by $T+L$ the $(\delta,\nu)$-tree obtained from $T$ by rotating the nodes $q_0,q_1,\dots ,q_{\ell-1}$ in $T$ in this order. An example of these concepts is illustrated in~\Cref{fig_horizontal_L}. 
		
		\begin{figure}[htb]
			\centering
			\begin{center}
				\centering
				\begin{minipage}[c]{.4\linewidth}
					\centering
					\def\svgwidth{.9\linewidth}
					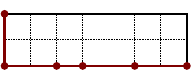
					
					$\left. T \right|_R$
				\end{minipage}
				\begin{minipage}[c]{.4\linewidth}
					\centering
					\def\svgwidth{.9\linewidth}
					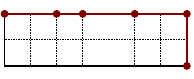
					
					$\left. (T+L) \right|_R$
				\end{minipage}
			\end{center}
			\caption{Schematic illustration of a horizontal $L$ and the tree $T+L$.}
			\label{fig_horizontal_L}
		\end{figure}
		
		\begin{lemma}\label{lem_horizontal_L}
			Let $L$ be a horizontal L of length $\ell$ of a $(\delta,\nu)$-tree $T$. Then, $[T,T+L]$ is a left interval of length $\ell$ in~$\altTamTrees{\nu}{\delta}$. 
			Moreover, every left interval of $\altTamTrees{\nu}{\delta}$ with bottom element $T$ is of this form.  
		\end{lemma}
		
		\begin{proof}
			This follows by the definition of left intervals.
		\end{proof}
		
		\begin{proposition}\label{prop_left_interval_counting}
			Let $T$ be a $(\delta,\nu)$-tree with row vector $r(T)=(r_0,\dots,r_n)$.
			The number of left intervals of length $\ell$
			with bottom element~$T$ in $\altTamTrees{\nu}{\delta}$ is equal to  
			\[
			|\{0\leq i \leq n-1:\ r_i\geq \ell\}|.
			\]
		\end{proposition}
		\begin{proof}
			By~\Cref{lem_horizontal_L}, the left intervals of length $\ell$ with bottom element $T$ are of the form $[T,T+L]$ where $L$ is a horizontal L of length $\ell$ of $T$.
			There is one such $L$ for each $r_i\geq \ell$ with $0\leq i\leq n-1$, where $q_0,\dots,q_\ell$ are the $\ell+1$ left most nodes of $T$ at height $i$ and $p$ is the parent of $q_0$ in $T$.  
		\end{proof}

		
		
		The previous two results,~\Cref{lem_horizontal_L} and~\Cref{prop_left_interval_counting}, characterize the left intervals in~$\altTamTrees{\nu}{\delta}$ with respect to the row vectors of $(\delta,\nu)$-trees. 
		Our next goal is to have a similar characterization for the right intervals with respect to certain column vectors. As we will see, column vectors are not enough for such a characterization, and we will need to consider a notion of reduced column vectors. Before going into that, we first introduce column vectors and present some of their properties.

		\subsection{Column vectors}
		Given a path $\nu$ from~$(0,0)$ to~$(m,n)$, the \emph{reversed path} $\nurev$ is the path from~$(0,0)$ to $(n,m)$ obtained by reading $\nu$ from right to left and replacing east steps by north steps and vice versa. Equivalently, $\nurev= (\nurev_0,\dots, \nurev_m)$ where~$\nurev_i$ is the number of north steps of the path $\nu$ in column $m-i$.
		For instance, if $\nu=ENEEN=(1,2,0)$ then $\nurev=ENNEN=(1,0,1,0)$.
		This notion is convenient to characterize column vectors. 
		
		In order to define the column vector of a $(\delta,\nu)$-tree, it is convenient to assign an order $j_0 \prec_\delta \dots \prec_\delta j_m$ to the columns of $L_{\delta,\nu}$, obtained by reading the columns from shortest to longest, from right to left, as illustrated in~\Cref{fig_columns_and_reduced_columns} (left). 
		See also the three examples in~\Cref{fig_column_order}. 
		
		\begin{figure}[htb]
			\centering
			\begin{center}
				\centering
				\def\svgwidth{.25\linewidth}
				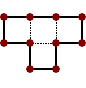
				\qquad\qquad
				\def\svgwidth{.25\linewidth}
\begingroup%
  \makeatletter%
  \providecommand\color[2][]{%
    \errmessage{(Inkscape) Color is used for the text in Inkscape, but the package 'color.sty' is not loaded}%
    \renewcommand\color[2][]{}%
  }%
  \providecommand\transparent[1]{%
    \errmessage{(Inkscape) Transparency is used (non-zero) for the text in Inkscape, but the package 'transparent.sty' is not loaded}%
    \renewcommand\transparent[1]{}%
  }%
  \providecommand\rotatebox[2]{#2}%
  \newcommand*\fsize{\dimexpr\f@size pt\relax}%
  \newcommand*\lineheight[1]{\fontsize{\fsize}{#1\fsize}\selectfont}%
  \ifx\svgwidth\undefined%
    \setlength{\unitlength}{24.75894477bp}%
    \ifx\svgscale\undefined%
      \relax%
    \else%
      \setlength{\unitlength}{\unitlength * \real{\svgscale}}%
    \fi%
  \else%
    \setlength{\unitlength}{\svgwidth}%
  \fi%
  \global\let\svgwidth\undefined%
  \global\let\svgscale\undefined%
  \makeatother%
  \begin{picture}(1,1.0150491)%
    \lineheight{1}%
    \setlength\tabcolsep{0pt}%
    \put(0,0){\includegraphics[width=\unitlength,page=1]{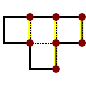}}%
    \put(0.95456104,0.90876204){\color[rgb]{0.50196078,0,0}\makebox(0,0)[t]{\lineheight{1.25}\smash{\begin{tabular}[t]{c}$\bar j_0$\end{tabular}}}}%
    \put(0.34871793,0.90876204){\color[rgb]{0.50196078,0,0}\makebox(0,0)[t]{\lineheight{1.25}\smash{\begin{tabular}[t]{c}$\bar j_1$\end{tabular}}}}%
    \put(0.65164079,0.90876204){\color[rgb]{0.50196078,0,0}\makebox(0,0)[t]{\lineheight{1.25}\smash{\begin{tabular}[t]{c}$\bar j_2$\end{tabular}}}}%
    \put(0.95456104,-0.00000044){\color[rgb]{0,0,0}\makebox(0,0)[t]{\lineheight{1.25}\smash{\begin{tabular}[t]{c}1\end{tabular}}}}%
    \put(0.65163993,-0.00000044){\color[rgb]{0,0,0}\makebox(0,0)[t]{\lineheight{1.25}\smash{\begin{tabular}[t]{c}2\end{tabular}}}}%
    \put(0.34871968,-0.00000044){\color[rgb]{0,0,0}\makebox(0,0)[t]{\lineheight{1.25}\smash{\begin{tabular}[t]{c}1\end{tabular}}}}%
    \put(0,0){\includegraphics[width=\unitlength,page=2]{columns-delta-reduced.pdf}}%
  \end{picture}%
\endgroup%

			\end{center}
			\caption{
				Left: the columns $j_0,j_1,j_2,j_3$ of $L_{\delta_\nu}$, their lengths are 1,1,2,2.
				Right: the reduced columns $\bar j_0, \bar j_1, \bar j_2$  of $L_{\delta_\nu}$ with lengths 1,1,2, where the relevant points are filled brown and the non-relevant points are unfilled green. 
				In both cases, the columns are read from shortest to longest, from rigth to left.
			}
			\label{fig_columns_and_reduced_columns}
		\end{figure}
		
		The \emph{column vector} of a $(\delta,\nu)$-tree $T$ is the vector 
		$$c_\delta(T)=(c_0,\dots,c_m),$$ 
		where $c_i+1$ is the number of nodes of $T$ in column $j_i$.  
		For instance, the three $(\delta,\nu)$-trees (for the three choices of $\delta$)  in~\Cref{fig_column_order}, all have column vector $(0,1,0,1)$. This means, in each of the cases, there are $0+1$ nodes of the tree in column $j_0$, $1+1$ nodes in column $j_1$, $0+1$ nodes in column $j_2$, and~$1+1$ nodes in column $j_3$.
		Equivalently, $c_i$ counts the number of edges of $T$ in column $j_i$. 
		
		\begin{figure}[htb]
			\centering
			\begin{center}
				\centering
				\def\svgwidth{.25\linewidth}
\begingroup%
  \makeatletter%
  \providecommand\color[2][]{%
    \errmessage{(Inkscape) Color is used for the text in Inkscape, but the package 'color.sty' is not loaded}%
    \renewcommand\color[2][]{}%
  }%
  \providecommand\transparent[1]{%
    \errmessage{(Inkscape) Transparency is used (non-zero) for the text in Inkscape, but the package 'transparent.sty' is not loaded}%
    \renewcommand\transparent[1]{}%
  }%
  \providecommand\rotatebox[2]{#2}%
  \newcommand*\fsize{\dimexpr\f@size pt\relax}%
  \newcommand*\lineheight[1]{\fontsize{\fsize}{#1\fsize}\selectfont}%
  \ifx\svgwidth\undefined%
    \setlength{\unitlength}{24.75002109bp}%
    \ifx\svgscale\undefined%
      \relax%
    \else%
      \setlength{\unitlength}{\unitlength * \real{\svgscale}}%
    \fi%
  \else%
    \setlength{\unitlength}{\svgwidth}%
  \fi%
  \global\let\svgwidth\undefined%
  \global\let\svgscale\undefined%
  \makeatother%
  \begin{picture}(1,0.83175681)%
    \lineheight{1}%
    \setlength\tabcolsep{0pt}%
    \put(0,0){\includegraphics[width=\unitlength,page=1]{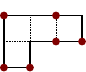}}%
    \put(0.95454467,0.7424244){\color[rgb]{0.50196078,0,0}\makebox(0,0)[t]{\lineheight{1.25}\smash{\begin{tabular}[t]{c}$j_0$\end{tabular}}}}%
    \put(0.34848313,0.7424244){\color[rgb]{0.50196078,0,0}\makebox(0,0)[t]{\lineheight{1.25}\smash{\begin{tabular}[t]{c}$j_2$\end{tabular}}}}%
    \put(0.04545454,0.7424244){\color[rgb]{0.50196078,0,0}\makebox(0,0)[t]{\lineheight{1.25}\smash{\begin{tabular}[t]{c}$j_3$\end{tabular}}}}%
    \put(0.65151521,0.7424244){\color[rgb]{0.50196078,0,0}\makebox(0,0)[t]{\lineheight{1.25}\smash{\begin{tabular}[t]{c}$j_1$\end{tabular}}}}%
    \put(0,0){\includegraphics[width=\unitlength,page=2]{tree-delta0.pdf}}%
  \end{picture}%
\endgroup%

				\qquad
				\def\svgwidth{.25\linewidth}
\begingroup%
  \makeatletter%
  \providecommand\color[2][]{%
    \errmessage{(Inkscape) Color is used for the text in Inkscape, but the package 'color.sty' is not loaded}%
    \renewcommand\color[2][]{}%
  }%
  \providecommand\transparent[1]{%
    \errmessage{(Inkscape) Transparency is used (non-zero) for the text in Inkscape, but the package 'transparent.sty' is not loaded}%
    \renewcommand\transparent[1]{}%
  }%
  \providecommand\rotatebox[2]{#2}%
  \newcommand*\fsize{\dimexpr\f@size pt\relax}%
  \newcommand*\lineheight[1]{\fontsize{\fsize}{#1\fsize}\selectfont}%
  \ifx\svgwidth\undefined%
    \setlength{\unitlength}{24.75002109bp}%
    \ifx\svgscale\undefined%
      \relax%
    \else%
      \setlength{\unitlength}{\unitlength * \real{\svgscale}}%
    \fi%
  \else%
    \setlength{\unitlength}{\svgwidth}%
  \fi%
  \global\let\svgwidth\undefined%
  \global\let\svgscale\undefined%
  \makeatother%
  \begin{picture}(1,0.83175681)%
    \lineheight{1}%
    \setlength\tabcolsep{0pt}%
    \put(0,0){\includegraphics[width=\unitlength,page=1]{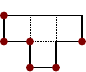}}%
    \put(0.95454467,0.7424244){\color[rgb]{0.50196078,0,0}\makebox(0,0)[t]{\lineheight{1.25}\smash{\begin{tabular}[t]{c}$j_0$\end{tabular}}}}%
    \put(0.34848313,0.7424244){\color[rgb]{0.50196078,0,0}\makebox(0,0)[t]{\lineheight{1.25}\smash{\begin{tabular}[t]{c}$j_3$\end{tabular}}}}%
    \put(0.04545454,0.7424244){\color[rgb]{0.50196078,0,0}\makebox(0,0)[t]{\lineheight{1.25}\smash{\begin{tabular}[t]{c}$j_1$\end{tabular}}}}%
    \put(0.65151521,0.7424244){\color[rgb]{0.50196078,0,0}\makebox(0,0)[t]{\lineheight{1.25}\smash{\begin{tabular}[t]{c}$j_2$\end{tabular}}}}%
    \put(0,0){\includegraphics[width=\unitlength,page=2]{tree-delta1.pdf}}%
  \end{picture}%
\endgroup%

				\qquad
				\def\svgwidth{.25\linewidth}
\begingroup%
  \makeatletter%
  \providecommand\color[2][]{%
    \errmessage{(Inkscape) Color is used for the text in Inkscape, but the package 'color.sty' is not loaded}%
    \renewcommand\color[2][]{}%
  }%
  \providecommand\transparent[1]{%
    \errmessage{(Inkscape) Transparency is used (non-zero) for the text in Inkscape, but the package 'transparent.sty' is not loaded}%
    \renewcommand\transparent[1]{}%
  }%
  \providecommand\rotatebox[2]{#2}%
  \newcommand*\fsize{\dimexpr\f@size pt\relax}%
  \newcommand*\lineheight[1]{\fontsize{\fsize}{#1\fsize}\selectfont}%
  \ifx\svgwidth\undefined%
    \setlength{\unitlength}{24.75002109bp}%
    \ifx\svgscale\undefined%
      \relax%
    \else%
      \setlength{\unitlength}{\unitlength * \real{\svgscale}}%
    \fi%
  \else%
    \setlength{\unitlength}{\svgwidth}%
  \fi%
  \global\let\svgwidth\undefined%
  \global\let\svgscale\undefined%
  \makeatother%
  \begin{picture}(1,0.83175681)%
    \lineheight{1}%
    \setlength\tabcolsep{0pt}%
    \put(0,0){\includegraphics[width=\unitlength,page=1]{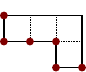}}%
    \put(0.95454467,0.7424244){\color[rgb]{0.50196078,0,0}\makebox(0,0)[t]{\lineheight{1.25}\smash{\begin{tabular}[t]{c}$j_2$\end{tabular}}}}%
    \put(0.34848313,0.7424244){\color[rgb]{0.50196078,0,0}\makebox(0,0)[t]{\lineheight{1.25}\smash{\begin{tabular}[t]{c}$j_0$\end{tabular}}}}%
    \put(0.04545454,0.7424244){\color[rgb]{0.50196078,0,0}\makebox(0,0)[t]{\lineheight{1.25}\smash{\begin{tabular}[t]{c}$j_1$\end{tabular}}}}%
    \put(0.65151521,0.7424244){\color[rgb]{0.50196078,0,0}\makebox(0,0)[t]{\lineheight{1.25}\smash{\begin{tabular}[t]{c}$j_3$\end{tabular}}}}%
    \put(0,0){\includegraphics[width=\unitlength,page=2]{tree-delta2.pdf}}%
  \end{picture}%
\endgroup%

			\end{center}
			\caption{The columns $j_0,j_1,j_2,j_3$ of $L_{\delta,\nu}$ for $\nu=ENEEN$ and the three possible choices of $\delta=(2,0),(1,0)$ and $(0,0)$. The columns are read from shortest to longest, from right to left.
				The column vector of the shown trees is $c_\delta(T)=(0,1,0,1)$ in all three cases.}
			\label{fig_column_order}
		\end{figure}

		\begin{proposition}\label{prop_column_vectors_characterization}
			A $(\delta,\nu)$-tree $T$ is completely characterized by its column vector. Moreover, $(c_0,\dots,c_m)$ is the column vector of some $(\delta,\nu)$-tree if and only if
			
			\begin{enumerate}
				\item $c_i\geq 0$ for all $i$, 
				\item $\sum_{i=0}^j c_i \leq \sum_{i=0}^j \nurev_i 
				$ for all $j$, and
				\item $\sum_{i=0}^m c_i =  \sum_{i=0}^m \nurev_i$.
			\end{enumerate}
		\end{proposition}
		
		We prove this proposition in several steps. 
		
		\begin{lemma}\label{lem_down_flushing}
			A $(\delta,\nu)$-tree $T$ can be reconstructed from its column vector. 
		\end{lemma}
		
		\begin{proof}
			Let  $ T $ be a $(\delta,\nu)$-tree. Label the elements of the tree $p_0,p_1,...,p_r$ from right to left, from bottom to top, as illustrated in~\Cref{fig_down_flushing}.
			
			We reconstruct $T$ recursively, by adding the elements $p_0,p_1,...,p_r$ one at a time in order. Note that if $p_i$ is not the top most element in its column, then all the lattice points on the left of $p_i$ are forbidden in the next steps, because they are incompatible with an element $p_j\in T$ that is above $p_i$ in the same column. 
			
			Now, when we add an element $p_j$ in the process of reconstructing $T$, then $p_j$ is necessarily located at the bottom most position of its column that is not forbidden by any element before. Otherwise, let $p_j\in T$ be the node with smallest label that does not satisfy that property, and let $q$ be the bottom most lattice point in the same column that is not forbidden by any element $p_i$ with $i<j$. In particular, $q$ is below~$p_j$ by assumption, and $q$ is compatible with every $p_i$ with $i<j$. Moreover, $q$ is also compatible with $p_k\in T$ with $k>j$, otherwise $p_k,p_j$ would be incompatible. So, we can add the element $q$ to $T$, creating a new compatible set, contradicting the maximality of $T$.  
			
			The tree $T$ can therefore be constructed by adding nodes from right to left, from bottom to top, avoiding forbidden positions. The forbidden positions are those to the left of a node that is not the top most node in a column. 
			The number of points in each column is determined by the column vector.
		\end{proof}
		
		\begin{figure}[htb]
			\centering
			\def\svgwidth{0.55\textwidth}
			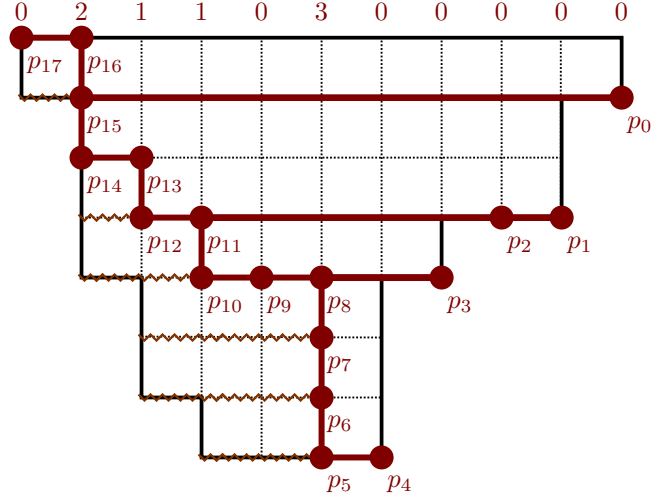
			\caption{The down flushing algorithm.}
			\label{fig_down_flushing}
		\end{figure}
		
		We call the algorithm described in the previous proof the \emph{down flushing algorithm}. Its input is a valid column vector $(c_0,\dots,c_m)$ (or the number of nodes in each column), and its output is the unique $(\delta,\nu)$-tree such that~$c_\delta(T)=(c_0,\dots,c_m)$. 
		Figure~\ref{fig_down_flushing} illustrate an example, where the labels on top represent the number of nodes, minus 1, in each column, and the forbidden positions are the ones that belong to the wiggly lines.

		\begin{lemma}\label{lem_vertical_small_change}
			Let $\delta,\delta'$ be two increment vectors with respect to $\nu$, such that $\delta'$ is obtained by either adding or subtracting 1 to one of the entries of~$\delta$. 
			For every $(\delta,\nu)$-tree $T$, there is a unique~$(\delta',\nu)$-tree~$T'$ such that $$c_\delta(T)=c_{\delta'}(T').$$   
		\end{lemma}
		\begin{proof}
			Uniqueness follows by~\Cref{lem_down_flushing}, so we just need to prove existence. 
			
			Let $T$ be a $(\delta,\nu)$-tree with column vector $c_\delta(T)=(c_0,\dots,c_m)$, and assume that $\delta'$ is obtained by subtracting 1 to a non-zero entry $\delta_a$ of~$\delta$. This operation produces a small transformation to the columns of $L_{\delta,\nu}$. All the columns of length larger than $n-a$ are moved one step to the right, while the subsequent column (of length $n-a$) is moved one step to their left. All other columns stay the same. The result is the new set $L_{\delta',\nu}$. An example is illustrated in~\Cref{fig_preserve_column_two}.
			
			Consider the labeling $j_0,\dots, j_m$ of the columns of $L_{\delta,\nu}$ (and also of the columns of $L_{\delta',\nu}$) obtained by reading the columns from shortest to longest, from right to left, as before. Assume that $j_{i_1}$ is the label of the column that was moved to the left under the small transformation that changes $\delta$ to $\delta'$. We also consider the columns $j_{i_2},\dots,j_{i_k}$, consisting of the columns of $L_{\delta,\nu}$, from right to left, of length bigger than $n-a$ that contain at least one node of $T$ at height bigger than or equal to $a$. The restriction of the tree $T$ to the nodes at height bigger than or equal to $a$ in columns $j_{i_1},\dots,j_{i_k}$ is marked as a bold red path on the left of~\Cref{fig_preserve_column_one,fig_preserve_column_two}. 
			It is a subpath of the unique path of the tree from column~$j_{i_1}$ to the root of the tree.  
			We will describe a small transformation to $T$ that produces a $(\delta',\nu)$-tree $T'$ with the same column vector as $T$. The result of this is illustrated on the right of~\Cref{fig_preserve_column_one,fig_preserve_column_two}, and affects the tree at the red marked nodes.
            The brown points in the columns between $\bar j_{i_k}$ and $\bar j_{i_1}$ are also moved one step to the right, together with their column. 
			
			Note that the columns $j_{i_2},\dots,j_{i_k}$ of $L_{\delta',\nu}$ are positioned one step to the right of columns $j_{i_2},\dots,j_{i_k}$ of $L_{\delta,\nu}$, while column $j_1$ was moved to some position to the left, see~\Cref{fig_preserve_column_one,fig_preserve_column_two}.
			
			Let $A$ be the set of rows that contain at least one node of the marked bold red path of $T$. We apply the following transformation to $T$. For each node $T$ in a column $j_{i_b}$, for $2\leq b\leq k$, that belongs to $A$, we draw a node in $T'$ in column $j_{i_b}$ but shifted down $c_{j_{i_1}}$ positions withing $A$. The $c_{j_{i_1}}+1$ nodes in column $j_{i_1}$ are moved to the top rows of $A$. 
			All other nodes of $T$ remain intact in their columns.  
			A schematic illustration of this transformation is shown in~\Cref{fig_preserve_column_one}, and an explicit example in~\Cref{fig_preserve_column_two}.
			
			The result is a $(\delta',\nu)$-tree $T'$ with the same column vector as $T$: $c_\delta(T)=c_{\delta'}(T')$.
			The reason why this procedure works is guarantied by a direct analysis of the down flushing algorithm. Moreover, we can also recover $T$ from $T'$ by a similar transformation in the reverse direction.
		\end{proof}

		\begin{figure}[htb]
			\centering
			\begin{center}
				\centering
				\def\svgwidth{0.45\textwidth}
\begingroup%
  \makeatletter%
  \providecommand\color[2][]{%
    \errmessage{(Inkscape) Color is used for the text in Inkscape, but the package 'color.sty' is not loaded}%
    \renewcommand\color[2][]{}%
  }%
  \providecommand\transparent[1]{%
    \errmessage{(Inkscape) Transparency is used (non-zero) for the text in Inkscape, but the package 'transparent.sty' is not loaded}%
    \renewcommand\transparent[1]{}%
  }%
  \providecommand\rotatebox[2]{#2}%
  \newcommand*\fsize{\dimexpr\f@size pt\relax}%
  \newcommand*\lineheight[1]{\fontsize{\fsize}{#1\fsize}\selectfont}%
  \ifx\svgwidth\undefined%
    \setlength{\unitlength}{127.5000036bp}%
    \ifx\svgscale\undefined%
      \relax%
    \else%
      \setlength{\unitlength}{\unitlength * \real{\svgscale}}%
    \fi%
  \else%
    \setlength{\unitlength}{\svgwidth}%
  \fi%
  \global\let\svgwidth\undefined%
  \global\let\svgscale\undefined%
  \makeatother%
  \begin{picture}(1,0.58762385)%
    \lineheight{1}%
    \setlength\tabcolsep{0pt}%
    \put(0,0){\includegraphics[width=\unitlength,page=1]{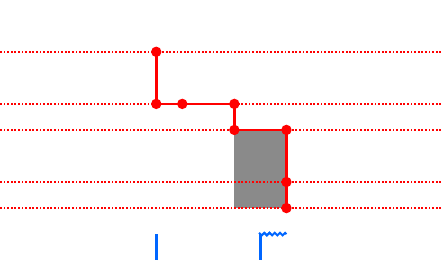}}%
    \put(0.71472955,0.0465883){\color[rgb]{0,0.4,1}\makebox(0,0)[lt]{\lineheight{1.25}\smash{\begin{tabular}[t]{l}$\delta_a \geq 1$\end{tabular}}}}%
    \put(0,0){\includegraphics[width=\unitlength,page=2]{small-transfo-columns-before.pdf}}%
    \put(0.64705889,0.55294149){\color[rgb]{1,0,0}\makebox(0,0)[t]{\lineheight{1.25}\smash{\begin{tabular}[t]{c}$j_{i_1}$\end{tabular}}}}%
    \put(0.52941189,0.55294149){\color[rgb]{1,0,0}\makebox(0,0)[t]{\lineheight{1.25}\smash{\begin{tabular}[t]{c}$j_{i_2}$\end{tabular}}}}%
    \put(0.41176472,0.55294149){\color[rgb]{1,0,0}\makebox(0,0)[t]{\lineheight{1.25}\smash{\begin{tabular}[t]{c}$j_{i_3}$\end{tabular}}}}%
    \put(0.35294132,0.55294149){\color[rgb]{1,0,0}\makebox(0,0)[t]{\lineheight{1.25}\smash{\begin{tabular}[t]{c}$j_{i_k}$\end{tabular}}}}%
    \put(0,0){\includegraphics[width=\unitlength,page=3]{small-transfo-columns-before.pdf}}%
  \end{picture}%
\endgroup%

				\qquad
				\def\svgwidth{0.45\textwidth}
\begingroup%
  \makeatletter%
  \providecommand\color[2][]{%
    \errmessage{(Inkscape) Color is used for the text in Inkscape, but the package 'color.sty' is not loaded}%
    \renewcommand\color[2][]{}%
  }%
  \providecommand\transparent[1]{%
    \errmessage{(Inkscape) Transparency is used (non-zero) for the text in Inkscape, but the package 'transparent.sty' is not loaded}%
    \renewcommand\transparent[1]{}%
  }%
  \providecommand\rotatebox[2]{#2}%
  \newcommand*\fsize{\dimexpr\f@size pt\relax}%
  \newcommand*\lineheight[1]{\fontsize{\fsize}{#1\fsize}\selectfont}%
  \ifx\svgwidth\undefined%
    \setlength{\unitlength}{127.5000036bp}%
    \ifx\svgscale\undefined%
      \relax%
    \else%
      \setlength{\unitlength}{\unitlength * \real{\svgscale}}%
    \fi%
  \else%
    \setlength{\unitlength}{\svgwidth}%
  \fi%
  \global\let\svgwidth\undefined%
  \global\let\svgscale\undefined%
  \makeatother%
  \begin{picture}(1,0.58762385)%
    \lineheight{1}%
    \setlength\tabcolsep{0pt}%
    \put(0,0){\includegraphics[width=\unitlength,page=1]{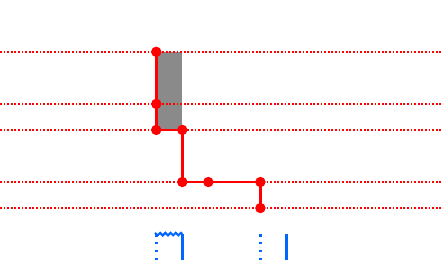}}%
    \put(0.71472921,0.04456473){\color[rgb]{0,0.4,1}\makebox(0,0)[lt]{\lineheight{1.25}\smash{\begin{tabular}[t]{l}$\delta_a-1$\end{tabular}}}}%
    \put(0,0){\includegraphics[width=\unitlength,page=2]{small-transfo-columns-after.pdf}}%
    \put(0.35294114,0.55294149){\color[rgb]{1,0,0}\makebox(0,0)[t]{\lineheight{1.25}\smash{\begin{tabular}[t]{c}$j_{i_1}$\end{tabular}}}}%
    \put(0.58823548,0.55294149){\color[rgb]{1,0,0}\makebox(0,0)[t]{\lineheight{1.25}\smash{\begin{tabular}[t]{c}$j_{i_2}$\end{tabular}}}}%
    \put(0.47058831,0.55294149){\color[rgb]{1,0,0}\makebox(0,0)[t]{\lineheight{1.25}\smash{\begin{tabular}[t]{c}$j_{i_3}$\end{tabular}}}}%
    \put(0.4117649,0.55294149){\color[rgb]{1,0,0}\makebox(0,0)[t]{\lineheight{1.25}\smash{\begin{tabular}[t]{c}$j_{i_k}$\end{tabular}}}}%
    \put(0,0){\includegraphics[width=\unitlength,page=3]{small-transfo-columns-after.pdf}}%
  \end{picture}%
\endgroup%

			\end{center}
			\caption{Schematic illustration of the transformation in the proof of~\Cref{lem_vertical_small_change}.}
			\label{fig_preserve_column_one}
		\end{figure}
		
		\begin{figure}[htb]
			\centering
			\begin{center}
				\centering
				\def\svgwidth{0.45\textwidth}
\begingroup%
  \makeatletter%
  \providecommand\color[2][]{%
    \errmessage{(Inkscape) Color is used for the text in Inkscape, but the package 'color.sty' is not loaded}%
    \renewcommand\color[2][]{}%
  }%
  \providecommand\transparent[1]{%
    \errmessage{(Inkscape) Transparency is used (non-zero) for the text in Inkscape, but the package 'transparent.sty' is not loaded}%
    \renewcommand\transparent[1]{}%
  }%
  \providecommand\rotatebox[2]{#2}%
  \newcommand*\fsize{\dimexpr\f@size pt\relax}%
  \newcommand*\lineheight[1]{\fontsize{\fsize}{#1\fsize}\selectfont}%
  \ifx\svgwidth\undefined%
    \setlength{\unitlength}{127.5000036bp}%
    \ifx\svgscale\undefined%
      \relax%
    \else%
      \setlength{\unitlength}{\unitlength * \real{\svgscale}}%
    \fi%
  \else%
    \setlength{\unitlength}{\svgwidth}%
  \fi%
  \global\let\svgwidth\undefined%
  \global\let\svgscale\undefined%
  \makeatother%
  \begin{picture}(1,0.77585924)%
    \lineheight{1}%
    \setlength\tabcolsep{0pt}%
    \put(0,0){\includegraphics[width=\unitlength,page=1]{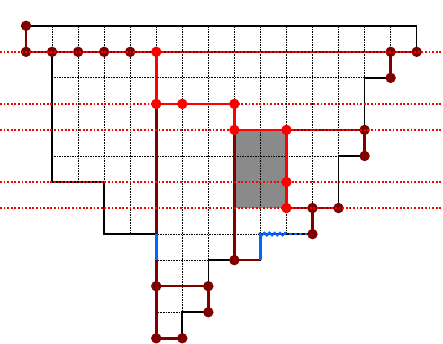}}%
    \put(0.80296448,0.23280012){\color[rgb]{0,0.4,1}\makebox(0,0)[lt]{\lineheight{1.25}\smash{\begin{tabular}[t]{l}$\delta_a=2$\end{tabular}}}}%
    \put(0,0){\includegraphics[width=\unitlength,page=2]{transfo-columns-before.pdf}}%
    \put(0.64705887,0.74117687){\color[rgb]{1,0,0}\makebox(0,0)[t]{\lineheight{1.25}\smash{\begin{tabular}[t]{c}$j_{i_1}$\end{tabular}}}}%
    \put(0.52941187,0.74117687){\color[rgb]{1,0,0}\makebox(0,0)[t]{\lineheight{1.25}\smash{\begin{tabular}[t]{c}$j_{i_2}$\end{tabular}}}}%
    \put(0.4117647,0.74117687){\color[rgb]{1,0,0}\makebox(0,0)[t]{\lineheight{1.25}\smash{\begin{tabular}[t]{c}$j_{i_3}$\end{tabular}}}}%
    \put(0.35294129,0.74117687){\color[rgb]{1,0,0}\makebox(0,0)[t]{\lineheight{1.25}\smash{\begin{tabular}[t]{c}$j_{i_4}$\end{tabular}}}}%
    \put(0,0){\includegraphics[width=\unitlength,page=3]{transfo-columns-before.pdf}}%
  \end{picture}%
\endgroup%

				\qquad
				\def\svgwidth{0.45\textwidth}
\begingroup%
  \makeatletter%
  \providecommand\color[2][]{%
    \errmessage{(Inkscape) Color is used for the text in Inkscape, but the package 'color.sty' is not loaded}%
    \renewcommand\color[2][]{}%
  }%
  \providecommand\transparent[1]{%
    \errmessage{(Inkscape) Transparency is used (non-zero) for the text in Inkscape, but the package 'transparent.sty' is not loaded}%
    \renewcommand\transparent[1]{}%
  }%
  \providecommand\rotatebox[2]{#2}%
  \newcommand*\fsize{\dimexpr\f@size pt\relax}%
  \newcommand*\lineheight[1]{\fontsize{\fsize}{#1\fsize}\selectfont}%
  \ifx\svgwidth\undefined%
    \setlength{\unitlength}{127.5000036bp}%
    \ifx\svgscale\undefined%
      \relax%
    \else%
      \setlength{\unitlength}{\unitlength * \real{\svgscale}}%
    \fi%
  \else%
    \setlength{\unitlength}{\svgwidth}%
  \fi%
  \global\let\svgwidth\undefined%
  \global\let\svgscale\undefined%
  \makeatother%
  \begin{picture}(1,0.77585924)%
    \lineheight{1}%
    \setlength\tabcolsep{0pt}%
    \put(0,0){\includegraphics[width=\unitlength,page=1]{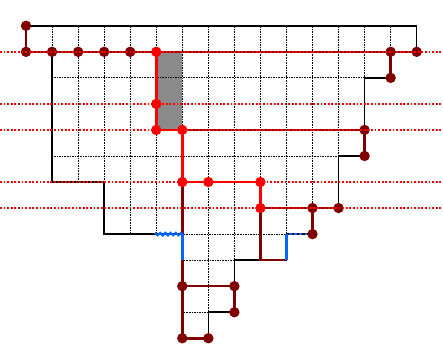}}%
    \put(0.80296447,0.23280012){\color[rgb]{0,0.4,1}\makebox(0,0)[lt]{\lineheight{1.25}\smash{\begin{tabular}[t]{l}$\delta'_a=2-1=1$\end{tabular}}}}%
    \put(0,0){\includegraphics[width=\unitlength,page=2]{transfo-columns-after.pdf}}%
    \put(0.35294112,0.74117789){\color[rgb]{1,0,0}\makebox(0,0)[t]{\lineheight{1.25}\smash{\begin{tabular}[t]{c}$j_{i_1}$\end{tabular}}}}%
    \put(0.58823545,0.74117789){\color[rgb]{1,0,0}\makebox(0,0)[t]{\lineheight{1.25}\smash{\begin{tabular}[t]{c}$j_{i_2}$\end{tabular}}}}%
    \put(0.47058829,0.74117789){\color[rgb]{1,0,0}\makebox(0,0)[t]{\lineheight{1.25}\smash{\begin{tabular}[t]{c}$j_{i_3}$\end{tabular}}}}%
    \put(0.41176487,0.74117789){\color[rgb]{1,0,0}\makebox(0,0)[t]{\lineheight{1.25}\smash{\begin{tabular}[t]{c}$j_{i_4}$\end{tabular}}}}%
  \end{picture}%
\endgroup%

			\end{center}
			\caption{Example of the transformation in the proof of~\Cref{lem_vertical_small_change}.}
			\label{fig_preserve_column_two}
		\end{figure}

		\begin{lemma}\label{lem_vertical_flushing}
			Let $\delta,\delta'$ be two increment vectors with respect to $\nu$.
			For every $(\delta,\nu)$-tree $T$, there is a unique~$(\delta',\nu)$-tree $T'$ such that $$c_\delta(T)=c_{\delta'}(T').$$   
		\end{lemma}
		\begin{proof}
			Any two increment vectors with respect to $\nu$ can be connected by a sequence of increment vectors, such that each vector is obtained from the previous one by either adding or subtracting 1 to one of its entries. The result then follows by~\Cref{lem_vertical_small_change}.  
		\end{proof}
		
		\begin{proof}[Proof of \Cref{prop_column_vectors_characterization}]
			A $(\delta,\nu)$-tree $T$ is completely characterized by its column vector by~\Cref{lem_down_flushing}. 
			Furthermore, the characterization of column vectors of $(\delta,\nu)$-trees is independent of the choice of increment vector $\delta$, by~\Cref{lem_vertical_flushing}. 
			So, we just need to prove the three conditions of the proposition for one particular choice of $\delta$. We choose the extreme case $\delta=\delta^{\max}$, where $\delta_i=\nu_i$. In this case $(\delta,\nu)$-trees are just the classical $\nu$-trees.  
			
			Classifying the column vectors of $\nu$-trees is the same as classifying the row vectors of $\nurev$-trees, because reversing the path transforms column vectors to row vectors, and vice versa. The three conditions of the proposition are then equivalent to the three conditions of~\Cref{prop_row_vectors_characterization} (for the extreme maximal case $\delta$).
		\end{proof}

		\subsection{Reduced column vectors}
		We say that a lattice point $p\in L_{\delta,\nu}$ is \emph{non-relevant} if it is the leftmost point of a row of $L_{\delta,\nu}$. 
		All other points in $L_{\delta,\nu}$ are called \emph{relevant}. 
		Figure~\ref{fig_columns_and_reduced_columns} (right) illustrates an example where the relevant points are filled brown, and the non-relevant points are unfilled green. 
		
		The \emph{reduced columns} are the columns of relevant points in $L_{\delta,\nu}$. These are shown in yellow in~\Cref{fig_columns_and_reduced_columns} (right).
		The three examples for $\nu=ENEEN$ and all possible choices of $\delta$ are shown in~\Cref{fig_column_order_reduced}. The reduced columns are colored yellow here as well for easier visualization.  
		
		In order to define the reduced column vector of a $(\delta,\nu)$-tree, it is convenient to assign an order $\overline j_0 \prec_\delta \dots \prec_\delta \overline j_{m-1}$ to the reduced columns of $F_{\delta,\nu}$, obtained by reading the reduced columns from shortest to longest, from right to left, as illustrated in~\Cref{fig_columns_and_reduced_columns} (right). See also the three examples in~\Cref{fig_column_order_reduced}.  
		
		The \emph{reduced column vector} of a $(\delta,\nu)$-tree $T$ is the vector 
		$$\overline c_\delta(T)=(\overline c_0,\dots,\overline c_{m-1}),$$ 
		where $\overline c_i+1$ is the number of nodes of $T$ in reduced column $\overline j_i$.  
		For instance, the three $(\delta,\nu)$-trees (for the three choices of $\delta$)  in~\Cref{fig_column_order_reduced}, all have reduced column vector $(0,1,0)$. This means, in each of the cases, there are $0+1$ nodes of the tree in reduced column $\overline j_0$, $1+1$ nodes in reduced column $\overline j_1$,  and $0+1$ nodes in reduced column $\overline j_2$. Note that the green nodes are non-relevant and do not belong to the reduced columns, and so are not counted here.
		
		\begin{figure}[htb]
			\centering
			\begin{center}
				\centering
				\def\svgwidth{.25\linewidth}
\begingroup%
  \makeatletter%
  \providecommand\color[2][]{%
    \errmessage{(Inkscape) Color is used for the text in Inkscape, but the package 'color.sty' is not loaded}%
    \renewcommand\color[2][]{}%
  }%
  \providecommand\transparent[1]{%
    \errmessage{(Inkscape) Transparency is used (non-zero) for the text in Inkscape, but the package 'transparent.sty' is not loaded}%
    \renewcommand\transparent[1]{}%
  }%
  \providecommand\rotatebox[2]{#2}%
  \newcommand*\fsize{\dimexpr\f@size pt\relax}%
  \newcommand*\lineheight[1]{\fontsize{\fsize}{#1\fsize}\selectfont}%
  \ifx\svgwidth\undefined%
    \setlength{\unitlength}{24.75008597bp}%
    \ifx\svgscale\undefined%
      \relax%
    \else%
      \setlength{\unitlength}{\unitlength * \real{\svgscale}}%
    \fi%
  \else%
    \setlength{\unitlength}{\svgwidth}%
  \fi%
  \global\let\svgwidth\undefined%
  \global\let\svgscale\undefined%
  \makeatother%
  \begin{picture}(1,0.8487477)%
    \lineheight{1}%
    \setlength\tabcolsep{0pt}%
    \put(0,0){\includegraphics[width=\unitlength,page=1]{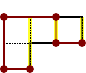}}%
    \put(0.65151351,0.74242145){\color[rgb]{0.50196078,0,0}\makebox(0,0)[t]{\lineheight{1.25}\smash{\begin{tabular}[t]{c}$\bar j_1$\end{tabular}}}}%
    \put(0.95454217,0.74242145){\color[rgb]{0.50196078,0,0}\makebox(0,0)[t]{\lineheight{1.25}\smash{\begin{tabular}[t]{c}$\bar j_0$\end{tabular}}}}%
    \put(0.34848396,0.74242145){\color[rgb]{0.50196078,0,0}\makebox(0,0)[t]{\lineheight{1.25}\smash{\begin{tabular}[t]{c}$\bar j_2$\end{tabular}}}}%
    \put(0,0){\includegraphics[width=\unitlength,page=2]{tree-delta0-reduced.pdf}}%
  \end{picture}%
\endgroup%

				\qquad
				\def\svgwidth{.25\linewidth}
\begingroup%
  \makeatletter%
  \providecommand\color[2][]{%
    \errmessage{(Inkscape) Color is used for the text in Inkscape, but the package 'color.sty' is not loaded}%
    \renewcommand\color[2][]{}%
  }%
  \providecommand\transparent[1]{%
    \errmessage{(Inkscape) Transparency is used (non-zero) for the text in Inkscape, but the package 'transparent.sty' is not loaded}%
    \renewcommand\transparent[1]{}%
  }%
  \providecommand\rotatebox[2]{#2}%
  \newcommand*\fsize{\dimexpr\f@size pt\relax}%
  \newcommand*\lineheight[1]{\fontsize{\fsize}{#1\fsize}\selectfont}%
  \ifx\svgwidth\undefined%
    \setlength{\unitlength}{24.75008597bp}%
    \ifx\svgscale\undefined%
      \relax%
    \else%
      \setlength{\unitlength}{\unitlength * \real{\svgscale}}%
    \fi%
  \else%
    \setlength{\unitlength}{\svgwidth}%
  \fi%
  \global\let\svgwidth\undefined%
  \global\let\svgscale\undefined%
  \makeatother%
  \begin{picture}(1,0.8487477)%
    \lineheight{1}%
    \setlength\tabcolsep{0pt}%
    \put(0,0){\includegraphics[width=\unitlength,page=1]{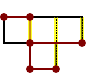}}%
    \put(0.65151351,0.74242145){\color[rgb]{0.50196078,0,0}\makebox(0,0)[t]{\lineheight{1.25}\smash{\begin{tabular}[t]{c}$\bar j_2$\end{tabular}}}}%
    \put(0.95454217,0.74242145){\color[rgb]{0.50196078,0,0}\makebox(0,0)[t]{\lineheight{1.25}\smash{\begin{tabular}[t]{c}$\bar j_0$\end{tabular}}}}%
    \put(0.34848396,0.74242145){\color[rgb]{0.50196078,0,0}\makebox(0,0)[t]{\lineheight{1.25}\smash{\begin{tabular}[t]{c}$\bar j_1$\end{tabular}}}}%
    \put(0,0){\includegraphics[width=\unitlength,page=2]{tree-delta1-reduced.pdf}}%
  \end{picture}%
\endgroup%

				\qquad
				\def\svgwidth{.25\linewidth}
\begingroup%
  \makeatletter%
  \providecommand\color[2][]{%
    \errmessage{(Inkscape) Color is used for the text in Inkscape, but the package 'color.sty' is not loaded}%
    \renewcommand\color[2][]{}%
  }%
  \providecommand\transparent[1]{%
    \errmessage{(Inkscape) Transparency is used (non-zero) for the text in Inkscape, but the package 'transparent.sty' is not loaded}%
    \renewcommand\transparent[1]{}%
  }%
  \providecommand\rotatebox[2]{#2}%
  \newcommand*\fsize{\dimexpr\f@size pt\relax}%
  \newcommand*\lineheight[1]{\fontsize{\fsize}{#1\fsize}\selectfont}%
  \ifx\svgwidth\undefined%
    \setlength{\unitlength}{24.75008597bp}%
    \ifx\svgscale\undefined%
      \relax%
    \else%
      \setlength{\unitlength}{\unitlength * \real{\svgscale}}%
    \fi%
  \else%
    \setlength{\unitlength}{\svgwidth}%
  \fi%
  \global\let\svgwidth\undefined%
  \global\let\svgscale\undefined%
  \makeatother%
  \begin{picture}(1,0.8487477)%
    \lineheight{1}%
    \setlength\tabcolsep{0pt}%
    \put(0,0){\includegraphics[width=\unitlength,page=1]{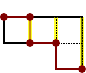}}%
    \put(0.65151351,0.74242145){\color[rgb]{0.50196078,0,0}\makebox(0,0)[t]{\lineheight{1.25}\smash{\begin{tabular}[t]{c}$\bar j_0$\end{tabular}}}}%
    \put(0.95454217,0.74242145){\color[rgb]{0.50196078,0,0}\makebox(0,0)[t]{\lineheight{1.25}\smash{\begin{tabular}[t]{c}$\bar j_2$\end{tabular}}}}%
    \put(0.34848396,0.74242145){\color[rgb]{0.50196078,0,0}\makebox(0,0)[t]{\lineheight{1.25}\smash{\begin{tabular}[t]{c}$\bar j_1$\end{tabular}}}}%
    \put(0,0){\includegraphics[width=\unitlength,page=2]{tree-delta2-reduced.pdf}}%
  \end{picture}%
\endgroup%

			\end{center}
			\caption{The ordering $\overline j_0\prec_\delta \dots\prec_\delta \overline j_2$ of the reduced columns of $L_{\delta,\nu}$ for \mbox{$\nu=ENEEN$} and the three possible choices of $\delta=(2,0),(1,0)$ and $(0,0)$. The reduced columns (colored yellow) are read from shortest to longest, from right to left.
				The reduced column vector of the shown trees is $\overline c_\delta(T)=(0,1,0)$ in all three cases.}
			\label{fig_column_order_reduced}
		\end{figure}

		\begin{proposition}\label{prop_reduced_column_vectors_characterization}
			A $(\delta,\nu)$-tree $T$ is completely characterized by its reduced column vector. Moreover, $(\overline c_0,\dots,\overline c_{m-1})$ is the reduced column vector of some $(\delta,\nu)$-tree if and only if
			
			\begin{enumerate}
				\item $\overline c_i\geq 0$ for all $i$, 
				\item $\sum_{i=0}^j \overline c_i \leq \sum_{i=0}^j \nurev_i 
				$ for all $j$.
			\end{enumerate}
		\end{proposition}
		
		The proof of this proposition follows the same steps as the proof of~\Cref{prop_column_vectors_characterization} for column vectors. 
		We write all the (somewhat repeated) details for self containment. 
		
		\begin{lemma}\label{lem_reduced_down_flushing}
			A $(\delta,\nu)$-tree $T$ can be reconstructed from its reduced column vector. 
		\end{lemma}
		
		\begin{proof}
			We proceed in a similar way as in the proof of~\Cref{lem_down_flushing}, with the small difference that we need to be careful what to do with the non-relevant positions, which are not counted by the reduced column vector. 
			
			Let $\bar c_{\delta}(T)=(\bar c_0, \dots,\bar c_{m-1})$ be the reduced column vector of $T$. Similarly as before, the tree $T$ can be reconstructed by adding nodes from right to left, from bottom to top, avoiding the forbidden positions that are to the left of a node that is not the top most node of its column. 
			Here comes the tricky part. When we want to add the nodes in column $\bar j_i$, there are two possible scenarios: 
			
			(1) If there are non-relevant positions in column $\bar j_i$ that are not forbidden by any of the nodes added before in the process, then these non-relevant positions are automatically compatible with all the nodes of the tree $T$ (the ones that were already added, and all the future ones). Therefore, all the non-relevant nodes in column $\bar j_i$ that are not forbidden by any previously added node should be added to $T$. After this we proceed adding $\bar c_i+1$ nodes from bottom to top in the positions that are not forbidden in column $\bar j_i$.  
			
			(2) If all the non-relevant positions in column $\bar j_i$ are forbidden, then we just proceed adding $\bar c_i+1$ nodes from bottom to top in the positions that are not forbidden in that column. 
			
			This procedure reconstructs the tree $T$ and only depends on the reduced column vector. 
			
			An example is illustrated in~\Cref{fig_reduced_down_flushing}. Note that the unfilled green point $p_{14}$ is non-relevant, and was forced to be added to $T$ because it is not forbidden by any of the previously added nodes $p_1,\dots, p_{13}$. At this step of the process, one proceeds adding the 1+1 relevant points $p_{15},p_{16}$ in that column, which are counted by the corresponding entry plus one of the reduced column vector.   
		\end{proof}
		
		\begin{figure}[htb]
			\centering
			\def\svgwidth{0.55\textwidth}
			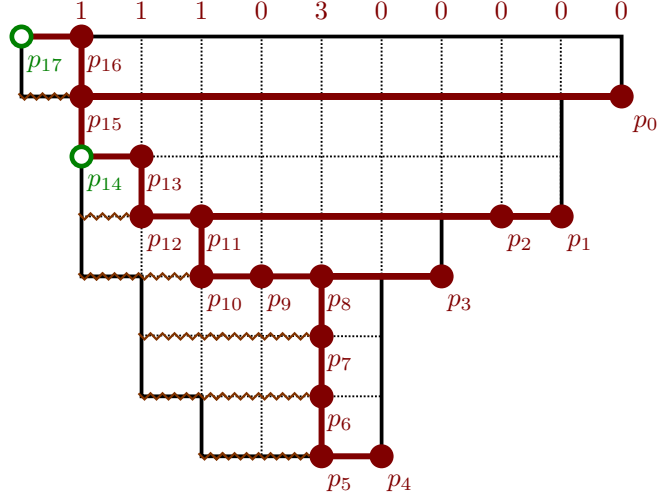
			\caption{The reduced down flushing algorithm.}
			\label{fig_reduced_down_flushing}
		\end{figure}

		We call the algorithm described in the previous proof the \emph{reduced down flushing algorithm}. Its input is a valid reduced column vector $(\overline c_0,\dots,\overline c_{m-1})$ (or the number of relevant nodes in each column), and its output is the unique $(\delta,\nu)$-tree such that~$\overline c_\delta(T)=(\overline c_0,\dots,\overline c_{m-1})$.

		\begin{lemma}\label{lem_vertical_small_change_reduced}
			Let $\delta,\delta'$ be two increment vectors with respect to $\nu$, such that $\delta'$ is obtained by either adding or subtracting 1 to one of the entries of~$\delta$. 
			For every $(\delta,\nu)$-tree $T$, there is a unique $(\delta',\nu)$-tree~$T'$ such that $$\overline c_\delta(T)=\overline c_{\delta'}(T').$$   
			Moreover, the heights of the non-relevant nodes of $T$ and $T'$ coincide. 
		\end{lemma}
		\begin{proof}
			Uniqueness follows by~\Cref{lem_reduced_down_flushing}, so we just need to prove existence. 
			
			Let $T$ be a $(\delta,\nu)$-tree with reduced column vector $\bar c_\delta(T)=(\bar c_0,\dots,\bar c_{m-1})$, and assume that $\delta'$ is obtained by subtracting 1 to a non-zero entry $\delta_a$ of~$\delta$. This operation produces a small transformation to the reduced columns of $L_{\delta,\nu}$ (which is sligthly different to the transformation in the proof on~\Cref{lem_vertical_small_change}). All the reduced columns of length larger than $n-a$ are moved one step to the right, while the subsequent reduced column (of length $n-a$) is moved one step to their left. All other columns stay the same.   
			An example is illustrated in~\Cref{fig_preserve_reduced_column_two}.
			Here, we have chosen the same example as in~\Cref{fig_preserve_column_two}, to highlight the differences with the transformation described in the proof of~\Cref{lem_vertical_small_change}.
			
			Consider the labeling $\bar j_0,\dots, \bar j_{m-1}$ of the reduced columns of $L_{\delta,\nu}$ (and also of the reduced columns of $L_{\delta',\nu}$) obtained by reading the reduced columns from shortest to longest, from right to left, as before. Assume that $\bar j_{i_1}$ is the label of the reduced column that was moved to the left under the small transformation that changes $\delta$ to $\delta'$. We also consider the columns $\bar j_{i_2},\dots,\bar j_{i_k}$, consisting of the reduced columns of $L_{\delta,\nu}$, from right to left, of length bigger than $n-a$ that contain at least one node of $T$ at height bigger than or equal to $a$. The restriction of the tree $T$ to the nodes at height bigger than or equal to $a$ in the reduced columns $\bar j_{i_1},\dots,\bar j_{i_k}$ is marked as a bold red path on the left of~\Cref{fig_preserve_reduced_column_one,fig_preserve_reduced_column_two}. 
			It is a subpath of the unique path of the tree from column $\bar j_{i_1}$ to the root of the tree.  
			
			Note that column $j_{i_4}$, of length bigger than $n-a$ in~\Cref{fig_preserve_column_two}, is now a reduced column of length~$n-a$. That is why there is no $\bar j_{i_4}$ in our example in~\Cref{fig_preserve_reduced_column_two}.    
			
			We will describe a small transformation to $T$ that produces a $(\delta',\nu)$-tree $T'$ with the same reduced column vector as $T$. The result of this is illustrated on the right of~\Cref{fig_preserve_reduced_column_one,fig_preserve_reduced_column_two}, and affects the red marked nodes of the tree.
            {The brown and green points between the columns $\bar j_{i_k}$ and $\bar j_{i_1}$ are also moved one step to the right.}. 
			
			Note that the columns $\bar j_{i_2},\dots,\bar j_{i_k}$ of $L_{\delta',\nu}$ are positioned one step to the right of columns $\bar j_{i_2},\dots,\bar j_{i_k}$ of $L_{\delta,\nu}$, while column $\bar j_1$ was moved to some position to the left, see~\Cref{fig_preserve_reduced_column_one,fig_preserve_reduced_column_two}.
			
			Let $A$ be the set of rows that contain at least one node of the marked bold red path of $T$. We apply the following transformation to $T$. For each node $T$ in a reduced column $\bar j_{i_b}$, for $2\leq b\leq k$, that belongs to $A$, we draw a node in $T'$ in the reduced column $\bar j_{i_b}$ but shifted down $\bar c_{j_{i_1}}$ positions withing~$A$. The $\bar c_{j_{i_1}}+1$ nodes in reduced column $\bar j_{i_1}$ are moved to the top rows of $A$. 
			All other relevant nodes of $T$ remain intact in their reduced columns, 
			and all non-relevant nodes remain intact in their ``not reduced'' columns. 
			A schematic illustration of this transformation is shown in~\Cref{fig_preserve_reduced_column_one}, and an explicit example in~\Cref{fig_preserve_reduced_column_two}.
			
			The result is a $(\delta',\nu)$-tree $T'$ with the same reduced column vector as $T$: $\bar c_\delta(T)=\bar c_{\delta'}(T')$, and such that the heights of the non-relevant nodes are preserved.
			The reason why this procedure works is guarantied by a direct analysis of the reduced down flushing algorithm. Moreover, we can also recover~$T$ from~$T'$ by a similar transformation in the reverse direction.
		\end{proof}
		
		\begin{figure}[htb]
			\centering
			\begin{center}
				\centering
				\def\svgwidth{.45\linewidth}
\begingroup%
  \makeatletter%
  \providecommand\color[2][]{%
    \errmessage{(Inkscape) Color is used for the text in Inkscape, but the package 'color.sty' is not loaded}%
    \renewcommand\color[2][]{}%
  }%
  \providecommand\transparent[1]{%
    \errmessage{(Inkscape) Transparency is used (non-zero) for the text in Inkscape, but the package 'transparent.sty' is not loaded}%
    \renewcommand\transparent[1]{}%
  }%
  \providecommand\rotatebox[2]{#2}%
  \newcommand*\fsize{\dimexpr\f@size pt\relax}%
  \newcommand*\lineheight[1]{\fontsize{\fsize}{#1\fsize}\selectfont}%
  \ifx\svgwidth\undefined%
    \setlength{\unitlength}{127.5000036bp}%
    \ifx\svgscale\undefined%
      \relax%
    \else%
      \setlength{\unitlength}{\unitlength * \real{\svgscale}}%
    \fi%
  \else%
    \setlength{\unitlength}{\svgwidth}%
  \fi%
  \global\let\svgwidth\undefined%
  \global\let\svgscale\undefined%
  \makeatother%
  \begin{picture}(1,0.66653925)%
    \lineheight{1}%
    \setlength\tabcolsep{0pt}%
    \put(0,0){\includegraphics[width=\unitlength,page=1]{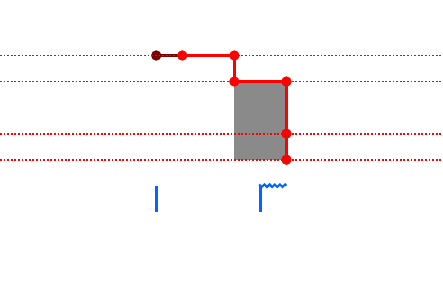}}%
    \put(0.71472989,0.23482331){\color[rgb]{0,0.4,1}\makebox(0,0)[lt]{\lineheight{1.25}\smash{\begin{tabular}[t]{l}$\delta_a \geq 1$\end{tabular}}}}%
    \put(0,0){\includegraphics[width=\unitlength,page=2]{small-transfo-reduced-columns-before.pdf}}%
    \put(0.6470589,0.62353136){\color[rgb]{1,0,0}\makebox(0,0)[t]{\lineheight{1.25}\smash{\begin{tabular}[t]{c}$\bar j_{i_1}$\end{tabular}}}}%
    \put(0.5294119,0.62353136){\color[rgb]{1,0,0}\makebox(0,0)[t]{\lineheight{1.25}\smash{\begin{tabular}[t]{c}$\bar j_{i_2}$\end{tabular}}}}%
    \put(0.41176473,0.62353136){\color[rgb]{1,0,0}\makebox(0,0)[t]{\lineheight{1.25}\smash{\begin{tabular}[t]{c}$\bar j_{i_k}$\end{tabular}}}}%
  \end{picture}%
\endgroup%

				\qquad
				\def\svgwidth{.45\linewidth}
\begingroup%
  \makeatletter%
  \providecommand\color[2][]{%
    \errmessage{(Inkscape) Color is used for the text in Inkscape, but the package 'color.sty' is not loaded}%
    \renewcommand\color[2][]{}%
  }%
  \providecommand\transparent[1]{%
    \errmessage{(Inkscape) Transparency is used (non-zero) for the text in Inkscape, but the package 'transparent.sty' is not loaded}%
    \renewcommand\transparent[1]{}%
  }%
  \providecommand\rotatebox[2]{#2}%
  \newcommand*\fsize{\dimexpr\f@size pt\relax}%
  \newcommand*\lineheight[1]{\fontsize{\fsize}{#1\fsize}\selectfont}%
  \ifx\svgwidth\undefined%
    \setlength{\unitlength}{127.5000036bp}%
    \ifx\svgscale\undefined%
      \relax%
    \else%
      \setlength{\unitlength}{\unitlength * \real{\svgscale}}%
    \fi%
  \else%
    \setlength{\unitlength}{\svgwidth}%
  \fi%
  \global\let\svgwidth\undefined%
  \global\let\svgscale\undefined%
  \makeatother%
  \begin{picture}(1,0.66653913)%
    \lineheight{1}%
    \setlength\tabcolsep{0pt}%
    \put(0,0){\includegraphics[width=\unitlength,page=1]{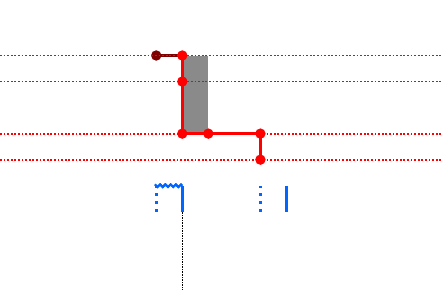}}%
    \put(0.71472972,0.23279972){\color[rgb]{0,0.4,1}\makebox(0,0)[lt]{\lineheight{1.25}\smash{\begin{tabular}[t]{l}$\delta_a-1$\end{tabular}}}}%
    \put(0,0){\includegraphics[width=\unitlength,page=2]{small-transfo-reduced-columns-after.pdf}}%
    \put(0.41176473,0.62353101){\color[rgb]{1,0,0}\makebox(0,0)[t]{\lineheight{1.25}\smash{\begin{tabular}[t]{c}$\bar j_{i_1}$\end{tabular}}}}%
    \put(0.58823549,0.62353101){\color[rgb]{1,0,0}\makebox(0,0)[t]{\lineheight{1.25}\smash{\begin{tabular}[t]{c}$\bar j_{i_2}$\end{tabular}}}}%
    \put(0.47058832,0.62353101){\color[rgb]{1,0,0}\makebox(0,0)[t]{\lineheight{1.25}\smash{\begin{tabular}[t]{c}$\bar j_{i_k}$\end{tabular}}}}%
    \put(0,0){\includegraphics[width=\unitlength,page=3]{small-transfo-reduced-columns-after.pdf}}%
  \end{picture}%
\endgroup%

			\end{center}
			\caption{Schematic illustration of the transformation in the proof of~\Cref{lem_vertical_small_change}.}
			\label{fig_preserve_reduced_column_one}
		\end{figure}
		
		\begin{figure}[htb]
			\centering
			\begin{center}
				\centering
				\def\svgwidth{.45\linewidth}
\begingroup%
  \makeatletter%
  \providecommand\color[2][]{%
    \errmessage{(Inkscape) Color is used for the text in Inkscape, but the package 'color.sty' is not loaded}%
    \renewcommand\color[2][]{}%
  }%
  \providecommand\transparent[1]{%
    \errmessage{(Inkscape) Transparency is used (non-zero) for the text in Inkscape, but the package 'transparent.sty' is not loaded}%
    \renewcommand\transparent[1]{}%
  }%
  \providecommand\rotatebox[2]{#2}%
  \newcommand*\fsize{\dimexpr\f@size pt\relax}%
  \newcommand*\lineheight[1]{\fontsize{\fsize}{#1\fsize}\selectfont}%
  \ifx\svgwidth\undefined%
    \setlength{\unitlength}{127.5000036bp}%
    \ifx\svgscale\undefined%
      \relax%
    \else%
      \setlength{\unitlength}{\unitlength * \real{\svgscale}}%
    \fi%
  \else%
    \setlength{\unitlength}{\svgwidth}%
  \fi%
  \global\let\svgwidth\undefined%
  \global\let\svgscale\undefined%
  \makeatother%
  \begin{picture}(1,0.78418727)%
    \lineheight{1}%
    \setlength\tabcolsep{0pt}%
    \put(0,0){\includegraphics[width=\unitlength,page=1]{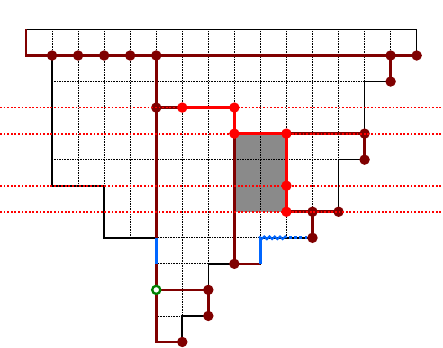}}%
    \put(0.80296484,0.23280007){\color[rgb]{0,0.4,1}\makebox(0,0)[lt]{\lineheight{1.25}\smash{\begin{tabular}[t]{l}$\delta_a=2$\end{tabular}}}}%
    \put(0,0){\includegraphics[width=\unitlength,page=2]{transfo-reduced-columns-before.pdf}}%
    \put(0.6470589,0.74117649){\color[rgb]{1,0,0}\makebox(0,0)[t]{\lineheight{1.25}\smash{\begin{tabular}[t]{c}$\bar j_{i_1}$\end{tabular}}}}%
    \put(0.5294119,0.74117649){\color[rgb]{1,0,0}\makebox(0,0)[t]{\lineheight{1.25}\smash{\begin{tabular}[t]{c}$\bar j_{i_2}$\end{tabular}}}}%
    \put(0.41176473,0.74117649){\color[rgb]{1,0,0}\makebox(0,0)[t]{\lineheight{1.25}\smash{\begin{tabular}[t]{c}$\bar j_{i_3}$\end{tabular}}}}%
    \put(0,0){\includegraphics[width=\unitlength,page=3]{transfo-reduced-columns-before.pdf}}%
  \end{picture}%
\endgroup%

				\qquad
				\def\svgwidth{.45\linewidth}
\begingroup%
  \makeatletter%
  \providecommand\color[2][]{%
    \errmessage{(Inkscape) Color is used for the text in Inkscape, but the package 'color.sty' is not loaded}%
    \renewcommand\color[2][]{}%
  }%
  \providecommand\transparent[1]{%
    \errmessage{(Inkscape) Transparency is used (non-zero) for the text in Inkscape, but the package 'transparent.sty' is not loaded}%
    \renewcommand\transparent[1]{}%
  }%
  \providecommand\rotatebox[2]{#2}%
  \newcommand*\fsize{\dimexpr\f@size pt\relax}%
  \newcommand*\lineheight[1]{\fontsize{\fsize}{#1\fsize}\selectfont}%
  \ifx\svgwidth\undefined%
    \setlength{\unitlength}{127.5000036bp}%
    \ifx\svgscale\undefined%
      \relax%
    \else%
      \setlength{\unitlength}{\unitlength * \real{\svgscale}}%
    \fi%
  \else%
    \setlength{\unitlength}{\svgwidth}%
  \fi%
  \global\let\svgwidth\undefined%
  \global\let\svgscale\undefined%
  \makeatother%
  \begin{picture}(1,0.78663286)%
    \lineheight{1}%
    \setlength\tabcolsep{0pt}%
    \put(0,0){\includegraphics[width=\unitlength,page=1]{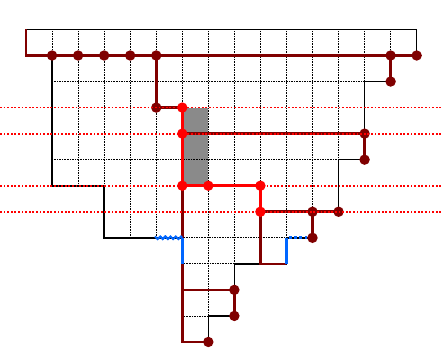}}%
    \put(0.80296484,0.23524566){\color[rgb]{0,0.4,1}\makebox(0,0)[lt]{\lineheight{1.25}\smash{\begin{tabular}[t]{l}$\delta_a=2-1=1$\end{tabular}}}}%
    \put(0,0){\includegraphics[width=\unitlength,page=2]{transfo-reduced-columns-after.pdf}}%
    \put(0.4117649,0.74362242){\color[rgb]{1,0,0}\makebox(0,0)[t]{\lineheight{1.25}\smash{\begin{tabular}[t]{c}$\bar j_{i_1}$\end{tabular}}}}%
    \put(0.58823566,0.74362242){\color[rgb]{1,0,0}\makebox(0,0)[t]{\lineheight{1.25}\smash{\begin{tabular}[t]{c}$\bar j_{i_2}$\end{tabular}}}}%
    \put(0.47058849,0.74362242){\color[rgb]{1,0,0}\makebox(0,0)[t]{\lineheight{1.25}\smash{\begin{tabular}[t]{c}$\bar j_{i_3}$\end{tabular}}}}%
    \put(0,0){\includegraphics[width=\unitlength,page=3]{transfo-reduced-columns-after.pdf}}%
  \end{picture}%
\endgroup%

			\end{center}
			\caption{Example of the transformation in the proof of~\Cref{lem_vertical_small_change}.}
			\label{fig_preserve_reduced_column_two}
		\end{figure}
		
		\begin{lemma}\label{lem_reduced_vertical_flushing}
			Let $\delta,\delta'$ be two increment vectors with respect to $\nu$.
			For every $(\delta,\nu)$-tree $T$, there is a unique $(\delta',\nu)$-tree $T'$ such that $$\overline c_\delta(T)=\overline c_{\delta'}(T').$$   
			Moreover, the heights of the non-relevant nodes of $T$ and $T'$ coincide. 
		\end{lemma}
		\begin{proof}
			Any two increment vectors with respect to $\nu$ can be connected by a sequence of increment vectors, such that each vector is obtained from the previous one by either adding or subtracting 1 to one of its entries. The result then follows by~\Cref{lem_vertical_small_change_reduced}.  
		\end{proof}
		
		\begin{proof}[Proof of \Cref{prop_reduced_column_vectors_characterization}]
			A $(\delta,\nu)$-tree $T$ is completely characterized by its reduced column vector by~\Cref{lem_reduced_down_flushing}. 
			Furthermore, the characterization of reduced column vectors of $(\delta,\nu)$-trees is independent of the choice of increment vector $\delta$, by~\Cref{lem_reduced_vertical_flushing}. 
			So, we just need to prove the two conditions of the proposition for one particular choice of $\delta$. We choose the extreme case $\delta=\delta^{\max}$, where $\delta_i=\nu_i$. In this case $(\delta,\nu)$-trees are just the classical $\nu$-trees.  
			
			The reduced column vector $(\overline c_0,\dots , \overline c_{m-1})$ of a $\nu$-tree $T$ is obtained from the row vector $(r_0,\dots,r_m)$ of the corresponding $\nurev$-tree $\reverse{T}$ by removing its last entry $r_m$. The two conditions of the proposition are then equivalent to the first two conditions of~\Cref{prop_row_vectors_characterization} (for the extreme maximal case $\delta$). 
            The third condition was about the number of points in the top row of $\reverse{T}$, which correspond to the non-relevant points in $T$.  
		\end{proof}
		
		\subsection{Reduced column vectors and right intervals}~
		We are finally ready to provide our characterization of right intervals in $\altTamTrees{\nu}{\delta}$ in terms of reduced column vectors. 
		
		Given a $(\delta,\nu)$-tree $T$, we say that an \emph{ordered} set 
		$L=\{p,q_0',q_1',\dots,q_\ell' \}\subseteq T$ is a \emph{vertical L} of $T$ if $L$ is the restriction of $T$ to a rectangle $R\subseteq F_{\delta,\nu}$ of the grid, such that $p$ is the top-left corner of $R$, and $q_0',q_1',\dots , q_\ell'$ appear in this order from top to bottom on the right side of $R$, with $q_0'$ being its top-right corner and  $q_\ell'$ its bottom-right corner. 
		Note that no other elements of $T$ belong to $R$. We say that the length of $L$ is equal to $\ell$. 
		We denote by $T-L$ the $(\delta,\nu)$-tree obtained from $T$ by rotating down the nodes $q_0',q_1',\dots ,q_{\ell-1}'$ in $T$ in this order. 
		
		Note that the condition $R\subseteq F_{\delta,\nu}$ is crucial here, to guaranty that the result after applying these rotations is still contained in the Ferrers diagram $F_{\delta,\nu}$, otherwise $T-L$ would not be a $(\delta,\nu)$-tree. 
		In particular, if $R\subseteq F_{\delta,\nu}$ then $q_0',q_1',\dots ,q_{\ell}'$ are all relevant nodes in $T$, and contribute to the reduced column vector.
        Vice versa, if $q_{\ell}'$ is relevant then $p \llcorner q'_\ell \in F_{\delta,\nu}$ because of the reduced down flushing algorithm, and thus $R\subseteq F_{\delta,\nu}$.
		
		An example of these concepts is illustrated in~\Cref{fig_vertical_L}. 
		
		\begin{figure}[htb]
			\centering
			\begin{center}
				\centering
				\qquad	\begin{minipage}[c]{.4\linewidth}
					\centering
					\def\svgwidth{.42\linewidth}
					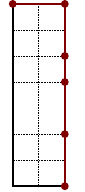
					
					$\left. T \right|_R$
				\end{minipage}
				\begin{minipage}[c]{.4\linewidth}
					\centering
					\def\svgwidth{.42\linewidth}
					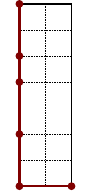
					
					$\left. (T-L) \right|_R$
				\end{minipage}\qquad
			\end{center}
			\caption{Schematic illustration of a vertical $L$ and the tree $T-L$.}
			\label{fig_vertical_L}
		\end{figure}
		
		\begin{lemma}\label{lem_vertical_L}
			Let $L$ be a vertical L of length $\ell$ of a $(\delta,\nu)$-tree $T$. Then, $[T-L,T]$ is a right interval of length $\ell$ in~$\altTamTrees{\nu}{\delta}$. 
			Moreover, every right interval of $\altTamTrees{\nu}{\delta}$ with top element $T$ is of this form.  
		\end{lemma}
		
		\begin{proof}
			This follows by the definition of right intervals. 
		\end{proof}

		\begin{proposition}\label{prop_right_interval_counting}
			Let $T$ be a $(\delta,\nu)$-tree with reduced column vector $\overline c_\delta(T)=(\overline c_0,\dots,\overline c_{m-1})$. 
			The number of right intervals of length $\ell$
			with top element~$T$ in $\altTamTrees{\nu}{\delta}$ is equal to  
			\[
			|\{0\leq i \leq m-1:\ \overline c_i\geq \ell\}|.
			\]
		\end{proposition}
		\begin{proof}
			By~\Cref{lem_vertical_L}, the right intervals of length $\ell$ with top element $T$ are of the form $[T-L,T]$ where $L$ is a vertical L of length $\ell$ of $T$.
			There is one such $L$ for each $\overline c_i\geq \ell$ with $0\leq i\leq m-1$, where $q_0',\dots,q_\ell'$ are the $\ell+1$ top most nodes of $T$ at column $\overline j_i$ and $p$ is the parent of $q_0$ in $T$.  
		\end{proof}
		
		\section{Bijections between linear intervals}\label{sec_bijections_linearintervals}
		Using the tools developed in the previous section, we are now ready to prove one of our main results.
		
		\begin{theorem}\label{thm_linear_intervals}
			For a fixed path $\nu$, all alt $\nu$-Tamari lattices $\altTam{\nu}{\delta}$ have the same number of linear intervals of length $\ell$.
		\end{theorem}
		
		This is a direct consequence of~\Cref{prop_linear_intervals_altnu} and~\Cref{cor_leftintervals,cor_rightintervals}, which show that the number of left intervals and the number of right intervals of length $\ell$ are preserved for any choice of $\delta$. Indeed, we prove more refined versions of these results in~\Cref{prop_leftintervals,prop_rightintervals}.
		
		\begin{remark}\label{rem:wrongcase} 
			If we chose any other path $\check{\nu}$ weakly below $\nu$ that does not satisfy $\check{\nu}_i \leq \nu_i$, for all~$i > 0$, then the restriction of $\Tam{\check{\nu}}$ to the subset of $\nu$-paths does not satisfy the enumerative result of \Cref{thm_linear_intervals}.
			
			More precisely, this poset still has the same number of left intervals (the left flushing argument presented afterwards still works) as all alt $\nu$-Tamari lattices. But based on computational experiments, it seems to have fewer right intervals. For instance, for $\nu = (1,2,0)$ and $\check{\nu} = (1,2,0)$, the distribution of linear intervals in the resulting poset is $(5, 5, 1)$ but the distribution of linear intervals in $\Tam{\nu}$ is~$(5, 5, 2)$.
		\end{remark}
	
	\begin{remark} 
		\Cref{thm_linear_intervals} generalizes the results obtained in~\cite{cheneviere_linear_2022} for the staircase $\nu=(NE)^n$.
		However, in this more general case, we usually do not have a closed formula counting the linear intervals of length $\ell$ similar to the one presented in~\cite{cheneviere_linear_2022}.
		
		In the $m$-Tamari lattice, where $\nu=(NE^m)^n$, one can adapt the decomposition given in~\cite{cheneviere_linear_2022} in order to find a closed formula for the number right intervals of length $\ell$: 
  $$\displaystyle m \binom{mn+n-\ell}{n-\ell-1}.$$ 
		We were not able to find a nice formula for the number of left intervals in this case.
		For $n=5$ and $m=2$, the distribution of left intervals in this lattice is $(728, 442, 222, 112, 47, 18, 5, 1)$. Since $47$ is a prime number, no such a nice product formula seems to exist.
	\end{remark}
		
		\subsection{The horizontal flushing and left intervals}
		We define the \emph{horizontal flushing} $\hflushing{\delta}{\delta'}$ as the map between the set of $(\delta,\nu)$-trees and the set of $(\delta',\nu)$-trees characterized by the property 
		\[
		\hflushing{\delta}{\delta'}(T) = T'
		\quad
		\longleftrightarrow
		\quad
		r(T)=r(T').
		\]
		That is, the map that preserves the row vector of the tree.
		This map is uniquely determined by this property, and can be computed as the composition 
		\[
		\hflushing{\delta}{\delta'}(T)=
		\flush_{\delta',\nu} \circ \flush_{\delta,\nu}^{-1},
		\]
		which sends a $(\delta,\nu)$-tree to the unique $\nu$-path with the same row vector, and then to the corresponding~$(\delta',\nu)$-tree. 
		In particular, $\hflushing{\delta}{\delta'}$ is a bijection, and can be described using a \emph{horizontal flushing algorithm}: 
		
		If $r(T)=(r_0,\dots,r_n)$, then $T'$ can be reconstructed by adding $r_i+1$ nodes, from bottom to top, from right to left, avoiding the forbidden positions that are above the nodes that are not the left most nodes in their row. 
		
		This gives a natural correspondence between the horizontal L's of $T$ and the horizontal L's of $T'$: an L of length $\ell$ in row $i$\footnote{here we mean that the bottom part of the L is in row $i$} of $T$ corresponds to the unique horizontal L of the same length in row $i$ of $T'$. By abuse of notation, we denote by $\hflushing{\delta}{\delta'}(L)=L'$ the horizontal~L of $T'$ associated to $L$, a horizontal~L of $T$.
		
		\begin{proposition}\label{prop_leftintervals}
			Let $T$ be a $(\delta,\nu)$-tree and $T'=\hflushing{\delta}{\delta'}(T)$ be its corresponding $(\delta',\nu)$-tree. 
			We also denote by $L'=\hflushing{\delta}{\delta'}(L)$ the horizontal L of $T'$ associated to $L$, a horizontal L of $T$. 
			\begin{enumerate}
				\item The number of left intervals of length $\ell$ in $\altTamTrees{\nu}{\delta}$ with bottom element $T$ is equal to the number of left intervals of length $\ell$ in $\altTamTrees{\nu}{\delta'}$ with bottom element $T'$.
				\item The map  
				\[
				[T,T+L] \rightarrow [T',T'+L']
				\]
				is a bijection between the left intervals of $\altTamTrees{\nu}{\delta}$ and the left intervals of $\altTamTrees{\nu}{\delta'}$.
			\end{enumerate}
		\end{proposition}
		
		\begin{proof}
			By~\Cref{prop_left_interval_counting}, the number of left intervals with bottom element $T$ depends only of the row vector $r(T)$. Since the $r(T)=r(T')$, then Item (1) follows. Item (2) is straight forward from the characterization of left intervals in~\Cref{lem_horizontal_L}.
		\end{proof}
		
		An example of the bijection between left intervals is illustrated in~\Cref{fig:BijectionLeftIntervals}. The maximal horizontal~L's are marked red for easier visualization. 
		
		
		\begin{figure}[h]
			\begin{center}
				\centering
				\def\svgwidth{0.4\textwidth}
				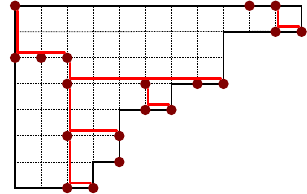
				\qquad
				\def\svgwidth{0.4\textwidth}
				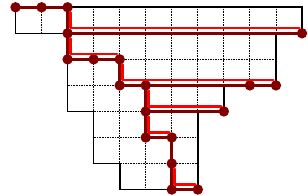
				\caption{Bijection between left intervals for 
					$\delta^{max}=(1,0,2,2,0,3,0)$ and $\delta = (0,0,1,2,0,1,0) $.
					Both trees have row vector $(1,0,1,1,3,2,1,2)$, whose entries plus one count the number of nodes in each of the rows.}
				\label{fig:BijectionLeftIntervals}
			\end{center}
		\end{figure}
		
		\begin{corollary}\label{cor_leftintervals}
			The number of left intervals  of length $\ell$ in~$\altTamTrees{\nu}{\delta}$ and $\altTamTrees{\nu}{\delta'}$ are the same.
		\end{corollary}
		\begin{proof}
			This is a direct consequence of~\Cref{prop_leftintervals}.
		\end{proof}

		\subsection{The reduced vertical flushing and right intervals}
		
		We define the \emph{vertical flushing} $\vflushing{\delta}{\delta'}$ as the map between the set of $(\delta,\nu)$-trees and the set of $(\delta',\nu)$-trees characterized by the property 
		\[
		\vflushing{\delta}{\delta'}(T) = T'
		\quad
		\longleftrightarrow
		\quad
		\bar c_{\delta}(T)=\bar c_{\delta'}(T').
		\]
		That is, the map that preserves the reduced column vector of the tree.

		This map is uniquely determined by this property by~\Cref{lem_reduced_vertical_flushing}. 
		In particular, $\vflushing{\delta}{\delta'}$ is a bijection, and can be described using a \emph{vertical flushing algorithm}: 
		
		If $\bar c_{\delta}(T)=(\bar c_0,\dots,\bar c_{m-1})$, then $T'$ can be reconstructed by adding nodes, from right to left, from bottom to top, avoiding the forbidden positions that are to the left of the nodes that are not the top most nodes in their column. 
		The difference here is that the number of nodes that we add to a column, whose reduced column is labeled $\bar j_i$, is not necessarily equal to $\bar c_{j_i}+1$: 
		we first add all the non-relevant nodes that are not forbidden by any of the previously added nodes; then we continue adding $\bar c_{j_i}+1$ relevant nodes from bottom to top in the non-forbidden available positions.
		
		This also gives a natural correspondence between the vertical L's of $T$ and the vertical L's of $T'$: an L of length $\ell$ in reduced column $\bar j_i$\footnote{here we mean that the right part of the L is in reduced column $\bar j_i$} of $T$ corresponds to the unique vertical L of the same length in reduced column $\bar j_i$ of $T'$. By abuse of notation, we denote by $\vflushing{\delta}{\delta'}(L)=L'$ the vertical L of $T'$ associated to $L$, a vertical L of $T$.
		
		\begin{proposition}\label{prop_rightintervals}
			Let $T$ be a $(\delta,\nu)$-tree and $T'=\vflushing{\delta}{\delta'}(T)$ be its corresponding $(\delta',\nu)$-tree. 
			We also denote by $L'=\vflushing{\delta}{\delta'}(L)$ the vertical L of $T'$ associated to $L$, a vertical L of $T$. 
			\begin{enumerate}
				\item The number of right intervals of length $\ell$ in $\altTamTrees{\nu}{\delta}$ with top element $T$ is equal to the number of right intervals of length $\ell$ in $\altTamTrees{\nu}{\delta'}$ with top element $T'$.
				\item The map  
				\[
				[T,T-L] \rightarrow [T',T'-L']
				\]
				is a bijection between the right intervals of $\altTamTrees{\nu}{\delta}$ and the right intervals of $\altTamTrees{\nu}{\delta'}$.
			\end{enumerate}
		\end{proposition}
		
		\begin{proof}
			By~\Cref{prop_right_interval_counting}, the number of right intervals with top element $T$ depends only of the reduced column vector $\bar c_{\delta}(T)$. Since the $\bar c_{\delta}(T)=\bar c_{\delta'}(T')$, then Item (1) follows. Item (2) is straight forward from the characterization of right intervals in~\Cref{lem_vertical_L}.
		\end{proof}
		
		Examples of the bijection between right intervals are illustrated in~\Cref{fig:BijectionRightIntervals,fig:BijectionRightIntervals2}. The maximal vertical L's are marked red for easier visualization. The green nodes are the non-relevant nodes.
		
		\begin{figure}[h]
			\begin{center}
				\centering
				\def\svgwidth{0.4\textwidth}
				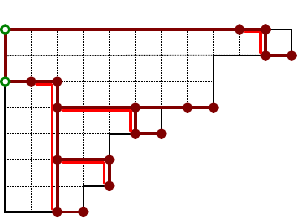
				\qquad
				\def\svgwidth{0.4\textwidth}
				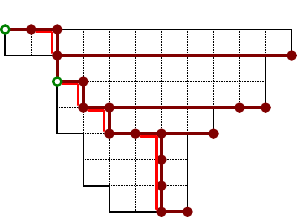
				\caption{Bijection between right intervals for 
					$\delta^{max}$ and $\delta = (0,0,1,2,0,1,0) $.
					Both trees have reduced column vector $(0,1,0,0,0,0,1,1,0,3,0)$, whose entries plus one count the number of relevant nodes in the reduced columns. The green non-relevant nodes are not counted.}
				\label{fig:BijectionRightIntervals}
			\end{center}
		\end{figure} 
		
		\begin{corollary}\label{cor_rightintervals}
			The number of right intervals of length $\ell$ in~$\altTamTrees{\nu}{\delta}$ and $\altTamTrees{\nu}{\delta'}$ are the same.
		\end{corollary}
		\begin{proof}
			This is a direct consequence of~\Cref{prop_rightintervals}.
		\end{proof}

		\begin{figure}[h]
			\begin{center}
				\begin{small}
				\centering
				\def\svgwidth{7.5cm}
				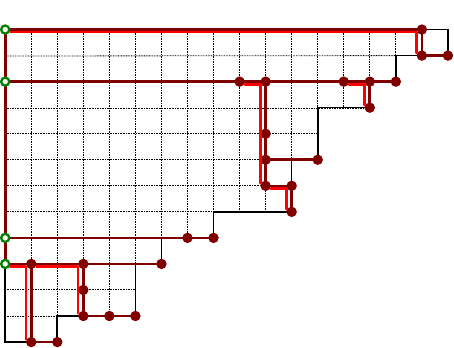
				
				\def\svgwidth{7.5cm}
				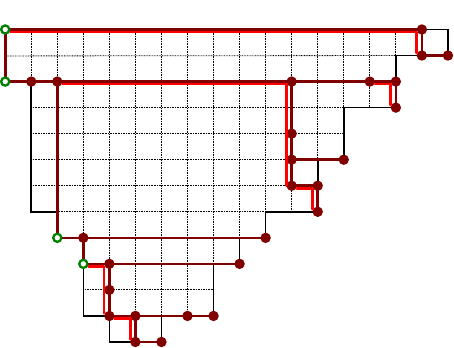
				\def\svgwidth{7.5cm}
				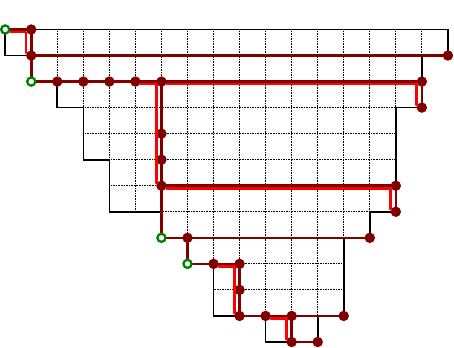
				\end{small}
				\caption{Bijection between right intervals for $\nu = (2,3,0,1,2,3,0,1,0,2,1,2,0)$, \\ $\delta^{max}$, 
					$\delta = (2,0,1,1,2,0,1,0,2,0,2,0) $ and $\delta' = (1,0,0,1,1,0,0,0,2,0,1,0)$. The three trees have reduced column vector $(0,1,0,1,0,0,1,3,0,0,0,0,0,0,2,0,1)$.}
				\label{fig:BijectionRightIntervals2}
			\end{center}
		\end{figure}

	\section*{Acknowledgement}
    This project started during discussions at the SLC 88 conference in Strobl in 2022, and we are grateful for the pleasant and motivating atmosphere during the conference. We are also grateful to Wenjie Fang, Matias von Bell, and especially Fr\'ed\'eric Chapoton for helpful and interesting discussions.

\bibliographystyle{alpha}
\bibliography{sample}

\end{document}